\documentclass[11pt, pagebackref]{amsart}

\usepackage{hyperref}

\makeatletter
\providecommand{\SG@adddot}{.}
\@ifundefined{href}{%
   \providecommand{\href}[2]{}%
   \renewcommand{\SG@adddot}{}}{}
\def\tospace#1{\@tospace#1 \tospace@delimiter}
\def\@tospace#1 #2\tospace@delimiter{#1}
\def\MR#1{\edef\MR@help{{http://www.ams.org/mathscinet-getitem?mr=\tospace{#1}}{\tospace{#1}}}%
\expandafter\href\MR@help\SG@adddot}
\makeatother

\providecommand*{\backref}{}
\providecommand*{\backrefalt}{}
\renewcommand*{\backref}[1]{}
\renewcommand*{\backrefalt}[4]{%
    \ifcase #1 %
    \or
      Cited page #2.
    \else
      Cited pages #2.
    \fi
}
\providecommand{\texorpdfstring}[2]{#1}

\newcommand{\R}{\mathbb{R}}
\newcommand{\C}{\mathbb{C}}
\newcommand{\dd}{\;{\rm d}}
\newcommand{\de}{{\rm d}}
\newcommand{\ic}{\mathbf{i}}
\newcommand{\reg}{s}

\newcommand{\N}{\mathbb{N}}

\DeclareMathOperator{\dist}{dist}
\DeclareMathOperator{\Real}{Re}

\newcommand{\Z}{\mathbb{Z}}
\newcommand{\st}{\;|\;}
\newcommand{\Transf}{\hat{T}}

\newcommand{\boL}{\mathcal{L}}
\newcommand{\Lp}{A}

\newcommand{\norm}[1]{\left\| #1 \right\|}

\DeclareMathOperator{\sgn}{sgn}
\newcommand{\D}{\mathcal{D}}
\newcommand{\boN}{\mathcal{N}}
\DeclareMathOperator{\Ima}{Im}

\newcommand{\BB}{\mathcal{B}}

\newcommand{\rad}{\varrho}
\newcommand{\uell}{{\underline{\ell}}}

\newcommand{\coloneqq}{\mathrel{\mathop:}=}

\newtheorem{thm}{Theorem}[section]
\newtheorem{prop}[thm]{Proposition}
\newtheorem{definition}[thm]{Definition}
\newtheorem{lem}[thm]{Lemma}
\newtheorem{cor}[thm]{Corollary}
\newtheorem*{prop*}{Proposition}

\theoremstyle{definition}

\newtheorem{rmk}[thm]{Remark}

\numberwithin{equation}{section}

\begin{document}

\title[Weak convergence in Gibbs-Markov maps]{Characterization of weak convergence of Birkhoff sums for Gibbs-Markov maps}
\author{S\'{e}bastien Gou\"{e}zel}

\address{IRMAR, CNRS UMR 6625,
Universit\'{e} de Rennes 1, 35042 Rennes, France}
\email{sebastien.gouezel@univ-rennes1.fr}

\begin{abstract}
We investigate limit theorems for Birkhoff sums of locally
H\"{o}lder functions under the iteration of Gibbs-Markov maps.
Aaronson and Denker have given sufficient conditions to have
limit theorems in this setting. We show that these conditions
are also necessary: there is no exotic limit theorem for
Gibbs-Markov maps. Our proofs, valid under very weak regularity
assumptions, involve weak perturbation theory and interpolation
spaces. For $L^2$ observables, we also obtain necessary and
sufficient conditions to control the speed of convergence in
the central limit theorem.
\end{abstract}
\date{March 6, 2009}
\maketitle

\section{Introduction and results}
Let $T$ be a probability preserving transformation on a space
$X$, and let $f:X\to \R$. We are interested in this paper in
limit theorems for sequences $(S_n f-A_n)/B_n$, where
$S_nf=\sum_{k=0}^{n-1}f\circ T^k$ and $A_n,B_n$ are real
numbers with $B_n>0$. If $T$ is a Gibbs-Markov map and $f$
satisfies a very weak regularity assumption, we will give
necessary and sufficient conditions for the convergence in
distribution of $(S_n f-A_n)/B_n$ to a nondegenerate random
variable. Sufficient conditions for this convergence are
already known by the work of Aaronson and Denker
\cite{aaronson_denker, aaronson_denker_central} (under stronger
regularity assumptions), and the main point of this article is
to show that these conditions are also necessary. We will also
considerably weaken the regularity assumptions of Aaronson and
Denker, by using weak perturbation theory
\cite{keller_liverani, herve_kl}.

Finding necessary conditions for limit theorems in dynamical
systems has already been considered in \cite{sarig_CPT}, but
here the author considered only random variables in a
controlled class of distributions, while our results apply to
all random variables. The paper \cite{cond_necessary} (see also
\cite{cond_necessary2}) gives in a wider setting (the condition
(B) in this paper is satisfied for Gibbs-Markov maps) a partial
answer to the questions we are considering: if one assumes that
$A_n=0$, then the limiting distribution has to be stable, as in
the case of i.i.d.~random variables. However, it does not
describe for which functions $f$ the convergence $S_n f/B_n \to
W$ takes place, nor does it treat the more difficult case
$A_n\not=0$.

At the heart of our argument lies a very precise control on the
leading eigenvalue of perturbed transfer operators: if the
function $f$ belongs to $L^p$ for $p\in (1,\infty)$, we obtain
such a control up to an error term $O(|t|^{p+\epsilon})$ for
some $\epsilon>0$, in Theorem \ref{thm_highdiff}. This estimate
is useful in a lot of different situations: we illustrate it by
deriving, in Appendix \ref{sec_app}, necessary and sufficient
conditions for the Berry-Esseen theorem (i.e., estimates on the
speed of convergence in the central limit theorem), for $L^2$
observables satisfying the same weak regularity condition as
above.

\subsection{The case of i.i.d.~random variables}
\label{subsecdecritD}

Since our limit theorems will be modeled on corresponding limit
theorems for sums of independent identically distributed random
variables, let us first describe the classical results in this
setting (the statements of this paragraph can be found in
\cite{feller_2} or \cite{ibragimov_linnik}).

\begin{definition}
Let $X_n$ be a sequence of random variable. This sequence
satisfies a \emph{nondegenerate limit theorem} if there exist
$A_n\in \R$ and $B_n>0$ such that $(X_n-A_n)/B_n$ converges in
distribution to a nonconstant random variable.
\end{definition}

\begin{definition}
\label{defSlowlyVar} A measurable function $L:\R_+^* \to \R_+^*$ is
\emph{slowly varying} if, for any $\lambda>0$, $L(\lambda
x)/L(x)\to 1$ when $x\to +\infty$.
\end{definition}

We define three sets of random variables as follows:
\begin{itemize}
\item Let $\D_1$ be the set of nonconstant random variables $Z$ whose
square is integrable.
\item Let $\D_2$ be the set of random variables $Z$ such that
the function $L(x)\coloneqq E( Z^2 1_{|Z|\leq x})$ is
unbounded and slowly varying (equivalently,
$P(|Z|>x)=x^{-2}\ell(x)$ for a function $\ell$ such that
$\tilde L(x)\coloneqq 2\int_1^x \frac{\ell(u)}{u}\dd u$ is
unbounded and slowly varying, and in this case $L$ and
$\tilde L$ are equivalent at $+\infty$).
\item Finally, let $\D_3$ be the set of random variables $Z$ such
that there exist $p\in (0,2)$, a slowly varying function
$L$ and $c_1,c_2\geq 0$ with $c_1+c_2=1$ such that
$P(Z>x)=(c_1+o(1)) L(x) x^{-p}$ and $P(Z<-x)=(c_2+o(1))
L(x)x^{-p}$ when $x\to +\infty$.
\end{itemize}
Let also $\D=\D_1\cup \D_2\cup \D_3$. The set $\D$ is exactly
the set of random variables satisfying nondegenerate limit
theorems, we will now describe the norming constants and the
limiting distribution in these theorems.

Let $Z\in \D$, let $Z_0,Z_1,\dots$ be i.i.d.~random variables
with the same distribution as $Z$. Then
\begin{itemize}
\item If $Z\in \D_1$, let $B_n=\sqrt{n}$
and $W=\boN(0, E(Z^2)-E(Z)^2)$.
\item If $Z\in \D_2$, let $B_n\to \infty$ satisfy $nL(B_n)\sim
B_n^2$ and let $W=\boN(0,1)$.
\item If $Z\in \D_3$, let $B_n\to\infty$ satisfy $nL(B_n) \sim
B_n^p$. Define $c=\Gamma(1-p)
\cos\left(\frac{p\pi}{2}\right)$ if $p\not=1$ and $c=\pi/2$
if $p=1$, and $\beta=c_1-c_2$. Let
$\omega(p,t)=\tan(p\pi/2)$ if $p\not=1$ and
$\omega(1,t)=-\frac{2}{\pi}\log |t|$. Let $W$ be the random
variable with characteristic function
  \begin{equation}
  E(e^{\ic tW})=e^{-c|t|^p(1-\ic\beta \sgn(t)\omega(p,t))}.
  \end{equation}
\end{itemize}

\begin{thm}
\label{thm:converge} In all three cases, there exists $A_n$
such that
  \begin{equation}
  \frac{\sum_{k=0}^{n-1} Z_k -A_n}{B_n} \to W.
  \end{equation}
One can take $A_n=nE(Z)$ if $Z$ is integrable, and $A_n=0$ if
$Z\in \D_3$ with $p<1$ (if $p=1$ but $Z$ is not integrable, the
value of $A_n$ is more complicated to express, see
\cite{aaronson_denker_independant}).

Moreover, the random variables in $\D$ are the only ones to
satisfy such a limit theorem: if a random variable $Z$ is such
that the sequence $\sum_{k=0}^{n-1} Z_k$ satisfies a
nondegenerate limit theorem, then $Z\in \D$.
\end{thm}
The set $\D$ can therefore be described as the set of random
variables belonging to a domain of attraction. The limit laws
in this theorem are the normal law and the so-called stable
laws. The two parts of this theorem are quite different: while
the direct implication is quite elementary (it boils down to a
computation of characteristic functions), the converse
implication, showing that a random variable automatically
belongs to $\D$ if it satisfies a nondegenerate limit theorem,
is much more involved, and requires the full strength of
L\'{e}vy-Khinchine theory.

The direct implication of Theorem \ref{thm:converge} describes
one limit theorem for random variables in $\D$, but does not
exclude the possibility of other limit theorems (for different
centering and scaling sequences). However, the following
convergence of types theorem (see e.g.~\cite[Theorem
14.2]{billingsley:book}) ensures that it can only be the case
in a trivial way:
\begin{thm}
\label{thmConvTypes}
Let $W_n$ be a sequence of random variables converging in
distribution to a nondegenerate random variable $W$. If, for
some $A_n\in \R$ and $B_n>0$, the sequence $(W_n-A_n)/B_n$ also
converges in distribution to a nondegenerate random variable
$W'$, then the sequences $A_n$ and $B_n$ converge respectively
to real numbers $A$ and $B$ (and $W'$ is equal in distribution
to $(W-A)/B$).
\end{thm}

The specific form of the convergence, the norming constants or
the limit laws in Theorem \ref{thm:converge} will not be
important to us. Indeed, we will prove in a dynamical setting
that Birkhoff sums satisfy a limit theorem if and only if the
sums of corresponding i.i.d.~random variables also satisfy a
limit theorem. Using Theorem \ref{thm:converge}, this will
readily imply a complete characterization of functions
satisfying a limit theorem -- it will not be necessary to look
into the details of Theorem \ref{thm:converge} and the specific
form of the domains of attraction, contrary to what is done in
\cite{aaronson_denker, aaronson_denker_central}.

\subsection{Limit theorems for Gibbs-Markov maps}

Let $(X,d)$ be a bounded metric space endowed with a
probability measure $m$. A probability preserving map $T:X\to
X$ is \emph{Gibbs-Markov} if there exists a partition $\alpha$
of $X$ (modulo $0$) by sets of positive measure, such that
\begin{enumerate}
\item Markov: for all $a \in \alpha$, $T(a)$ is a union (modulo $0$) of
elements of $\alpha$ and $T:a \to T(a)$ is invertible.
\item Big image and preimage property: there exists a
subset $\{a_1,\dots,a_n\}$ of $\alpha$ with the following
property: for any $a\in \alpha$, there exist
$i,j\in\{1,\dots,n\}$ such that $a \subset T(a_i)$ and $a_j
\subset T(a)$ (modulo $0$).
\item Expansion: there exists $\gamma<1$ such that for all $a\in
\alpha$, for almost all $x,y\in a$, $d(Tx,Ty) \geq
\gamma^{-1} d(x,y)$.
\item Distortion: for $a\in \alpha$, let $g$ be the inverse of the
jacobian of $T$ on $a$, i.e., $g(x)=\frac{ \dd m_{|a}}{\dd
(m\circ T_{|a})} (x)$ for $x\in a$. Then there exists $C$
such that, for all $a\in \alpha$, for almost all $x,y\in
a$,
  $
  \left| 1-\frac{g(x)}{g(y)} \right| \le C d(Tx,Ty)$.
\end{enumerate}
A Gibbs-Markov map is \emph{mixing} if, for all $a,b\in
\alpha$, there exists $N$ such that $b\subset T^n(a) \mod 0$
for any $n>N$. Since the general case reduces to the mixing
one, we will only consider mixing Gibbs-Markov maps.

For $f:X\to \R$ and $A\subset X$, let $Df(A)$ denote the best
Lipschitz constant of $f$ on $A$. If $f$ is integrable, we will
write $\int f$ or $E(f)$ for $\int f\dd m$, the reference
measure being always $\de m$. Our main result follows.
\begin{thm}
\label{thm:main} Let $T:X\to X$ be a mixing probability
preserving Gibbs-Markov map, and let $f:X\to \R$ satisfy
$\sum_{a\in\alpha} m(a)Df(a)^\eta<\infty$ for some $\eta\in
(0,1]$.

Assume $f\in L^2$. Then
\begin{itemize}
\item Either $f$ is the sum of a measurable
coboundary and a constant, i.e., there exist a measurable
function $u$ and a real number $c$ such that $f=u-u\circ
T+c$ almost everywhere. Then $u$ is bounded, and $S_n f-nc$
converges in distribution to the difference $Z-Z'$ where
$Z$ and $Z'$ are independent random variables with the same
distribution as $u$.
\item Otherwise, let $\tilde f=f-\int f\dd m$,
and define $\sigma^2=\int \tilde f^2+2\sum_{k=1}^\infty
\int \tilde f\cdot \tilde f\circ T^k$. Then this series
converges, $\sigma^2>0$, and $(S_n f - n\int f)/\sqrt{n}$
converges in distribution to $\boN(0,\sigma^2)$.
\end{itemize}

Assume that $f$ does not belong to $L^2$. Let $Z_0,Z_1,\dots$
be i.i.d.~random variables with the same distribution as $f$.
Consider sequences $A_n\in \R$ and $B_n>0$, and a nondegenerate
random variable $W$. Then $(S_nf -A_n)/B_n$ converges to $W$ if
and only if $(\sum_{k=0}^{n-1} Z_k -A_n)/B_n$ converges to $W$.
\end{thm}
In particular, it follows from the classification in the
i.i.d.~case that the Birkhoff sums of a function $f$ satisfy a
nondegenerate limit theorem if and only if the distribution of
$f$ belongs to the class $\D$ described in
Paragraph~\ref{subsecdecritD}.

In the $L^2$ case, the behavior of Birkhoff sums can be quite
different from the i.i.d.~case (see the formula for $\sigma^2$,
encompassing the interactions between different times). On the
other hand, when $f\not\in L^2$, the behavior is exactly the
same as in the i.i.d.~case (the interactions are negligible
with respect to the growth of the sums), and the good scaling
coefficients can be read directly from the independent case
Theorem~\ref{thm:converge}.

The ``sufficiency'' part of the theorem (i.e., the convergence
of the Birkhoff sums if $f$ is in the domain of attraction of a
gaussian or stable law) is known under stronger regularity
assumptions: if the function $f$ is locally H\"{o}lder continuous
(i.e.~$\sup_{a\in\alpha}D f(a)<\infty$), then the result is
proved in \cite{aaronson_denker, aaronson_denker_central} for
$f\not\in L^2$, and it follows from the classical Nagaev method
(see e.g.~\cite{rousseau-egele, guivarch-hardy} for subshifts
of finite type) when $f\in L^2$. The article
\cite{gouezel_stable} proves the same results under the
slightly weaker assumption $\sum m(a)Df(a)<\infty$. However,
these methods are not sufficient to deal with the weaker
assumption $\sum m(a)Df(a)^\eta<\infty$, hence new arguments
will be required to prove the sufficiency part of Theorem
\ref{thm:main}. The main difficulty is the following: even if
$f$ belongs to all $L^p$ spaces and $\sum
m(a)Df(a)^\eta<\infty$, it is possible that $\Transf f$ is not
locally H\"{o}lder continuous, in the sense that there exists
$a\in\alpha$ with $D(\Transf f)(a)=\infty$ (here, $\Transf$
denotes the transfer operator associated to $T$).\footnote{
This is for instance the case if $T$ is the full Markov shift
on infinitely many symbols $a_0,a_1,\dots$ with $m(a_i)=C
e^{-i^2/2}$, and $f$ vanishes on $[a_i]$ but on the set
$\bigcap_{n=0}^{i-1}T^{-n}(a_i)$, where it is equal to
$e^{i^2}$.}

However, the main novelty of the previous theorem is the
necessity part, showing that no exotic limit theorem can hold
for Gibbs-Markov maps, even if one assumes only very weak
regularity of the observable.

\begin{rmk}
The regularity condition $\sum_{a\in\alpha}
m(a)Df(a)^\eta<\infty$ is weaker than the conditions usually
encountered in the literature, but it appears in some natural
examples: for instance, if one tries to prove limit theorems
for the observable $f_0(x)=x^{-a}$ under the iteration of
$x\mapsto 2x \mod 1$, by inducing on $[1/2,1]$, then the
resulting induced observable $f$ satisfies such a condition for
some $\eta=\eta(a)<1$, but not for $\eta=1$.
\end{rmk}

\begin{rmk}
For general transitive Gibbs-Markov maps (without the mixing
assumption), it is still possible to prove that, if the
Birkhoff sums $S_n f$ of a function $f$ (with $\sum
m(a)Df(a)^\eta<\infty$) satisfy a nondegenerate limit theorem,
then the distribution of $f$ belongs to the class $\D$: the
proof we shall give below also applies to merely transitive
maps. However, the converse is not true. More precisely,
functions in $\D$ which are not the sum of a coboundary and a
constant satisfy a limit theorem, just like in the mixing case
(this follows readily from the mixing case), but this is not
the case in general for coboundaries.
\end{rmk}

\subsection{A more general setting}

Our results on Gibbs--Markov maps will be a consequence of a
more general theorem, making it possible to obtain necessary
and sufficient conditions for limit theorems of Birkhoff sums
whenever one can obtain sufficiently precise information on
characteristic functions.

\begin{definition}
\label{def_car} Let $T:X\to X$ be a probability preserving
mixing map, and let $f:X\to \R$ be measurable. The function $f$
admits a \emph{characteristic expansion} if there exist a
neighborhood $I$ of $0$ in $\R$, two measurable functions
$\lambda,\mu:I\to \R$ continuous at $0$ with
$\lambda(0)=\mu(0)=1$, and a sequence $\epsilon_n$ tending to
$0$ such that, for any $t\in I$ and any $n\in \N$,
  \begin{equation}
  \left| E(e^{\ic tS_n f}) -\lambda(t)^n \mu(t) \right|\leq \epsilon_n.
  \end{equation}
This characteristic expansion is \emph{accurate} if one of the
following properties holds:
\begin{itemize}
\item Either there exist $q\leq 2$ and $\epsilon>0$ such that $f\not\in L^q$ and
  \begin{equation}
  \label{eqPasL2}
  \lambda(t)=E(e^{\ic tf})+O(|t|^{q+\epsilon})+O(t^2)+o\left( \int
|e^{\ic tf}-1|^2\right)
  +O\left(\int |e^{\ic tf}-1|\right)^2.
  \end{equation}
\item Or $f\in L^2$ and there exists $c\in \C$ such that
  \begin{equation}
  \label{eqL2}
  \lambda(t)=1+ \ic tE(f)-c t^2/2 + o(t^2).
  \end{equation}
\end{itemize}
\end{definition}
When $f\not\in L^2$, this definition tells that a
characteristic expansion is accurate if $\lambda(t)$ is close
to $E(e^{\ic tf})$, up to error terms described by
\eqref{eqPasL2}. It should be noted that these error terms are
\emph{not} always negligible with respect to $1-E(e^{\ic tf})$,
but they are nevertheless sufficiently small (for sufficiently
many values of $t$) to ensure a good behavior, as shown by the
following theorem.

\begin{thm}
\label{thm_main_general} Let $T:X\to X$ be a probability
preserving mixing map, and let $f:X\to \R$ admit an accurate
characteristic expansion.

Assume that $f\in L^2$. Let $\lambda(t)=1+\ic tE(f)-c
t^2/2+o(t^2)$ be the characteristic expansion of $f$. Then
$\sigma^2\coloneqq c-E(f)^2\geq 0$, and $(S_n
f-nE(f))/\sqrt{n}$ converges in distribution to
$\boN(0,\sigma^2)$.

Assume that $f\not\in L^2$. Let $Z_0,Z_1,\dots$ be
i.i.d.~random variables with the same distribution as $f$.
Consider sequences $A_n\in \R$ and $B_n>0$, and a nondegenerate
random variable $W$. Then $(S_nf -A_n)/B_n$ converges to $W$ if
and only if $(\sum_{k=0}^{n-1} Z_k -A_n)/B_n$ converges to $W$.
\end{thm}

The flavor of this theorem is very similar to
Theorem~\ref{thm:main}. The only difference is in the $L^2$
case, when $\sigma^2=0$: Theorem~\ref{thm_main_general} only
says that $(S_n f -nE(f))/\sqrt{n}$ converges in distribution
to $0$ (note that this is a \emph{degenerate} limit theorem)
while Theorem~\ref{thm:main} gives a more precise conclusion in
this case, showing that $S_n f -nE(f)$ converges in
distribution to a nontrivial random variable. To get this
conclusion, one needs to show that a function $f$ satisfying
$\sigma^2=0$ is a coboundary -- this is indeed the case for
Gibbs-Markov maps, as we will see in
Paragraph~\ref{par_proof_GM_cobord}.

To deduce Theorem~\ref{thm:main} from
Theorem~\ref{thm_main_general}, we should of course check the
assumptions of the latter theorem. The following proposition is
therefore the core of our argument concerning Gibbs-Markov
maps.

\begin{prop}
\label{propGM} Let $T$ be a mixing Gibbs-Markov map, and let
$f:X\to \R$ satisfy $\sum_{a\in\alpha} m(a)Df(a)^\eta<\infty$
for some $\eta>0$. Then $f$ admits an accurate characteristic
expansion.
\end{prop}

\begin{rmk}
If we strengthened Definition \ref{def_car}, by requiring for
instance $\lambda(t)=E(e^{\ic tf})+O(t^2)$ when $f\not\in L^2$,
then Theorem~\ref{thm_main_general} would be much easier to
prove. However, we would not be able to prove Proposition
\ref{propGM} with this stronger definition. The form of the
error term in \eqref{eqPasL2} is the result of a tradeoff
between what is sufficient to prove
Theorem~\ref{thm_main_general}, and what we can prove for
Gibbs-Markov maps.
\end{rmk}

The rest of the paper is organized as follows:
Section~\ref{sec_AbstractProof} is devoted to the proof of
Theorem~\ref{thm_main_general}, using general considerations on
characteristic functions, while the results concerning
Gibbs-Markov maps (Proposition~\ref{propGM} and
Theorem~\ref{thm:main}) are proved in
Section~\ref{sec_proof_GM}. The required characteristic
expansion is obtained in some cases using classical
perturbation theory as in \cite{aaronson_denker}, but other
tools are also required in other cases: weak perturbation
theory \cite{keller_liverani, gouezel_liverani, herve_pene} and
interpolation spaces \cite{bergh_lofstrom_interpolation}.
Finally, Appendix \ref{sec_app} describes another application
of our techniques, to the speed in the central limit theorem.

\section{Using accurate characteristic expansions}
\label{sec_AbstractProof}

In this section, we prove Theorem~\ref{thm_main_general}. Let
$f$ be a function satisfying an accurate characteristic
expansion.

Assume first that $f$ is square integrable. Let $t\in \R$. If
$n$ is large enough, $t/\sqrt{n}$ belongs to the domain of
definition of $\lambda$, and
  \begin{align*}
  \lambda\left(\frac{t}{\sqrt{n}}\right)^n
  &=\exp \Bigl( n\log \bigl[1+\ic t E(f)/\sqrt{n} -c t^2/(2n) +
o(1/n)\bigr]\Bigr)
  \\
  &=\exp\Bigl( n \bigl[\ic tE(f)/\sqrt{n}-c t^2/(2n)+t^2
E(f)^2/(2n)+o(1/n)\bigr] \Bigr).
  \end{align*}
Hence, $e^{-\ic t\sqrt{n}E(f)} \lambda(t/\sqrt{n})^n$ converges
to $e^{ - (c -E(f)^2) t^2/2}$. By definition of a
characteristic expansion, this implies that $e^{-\ic
t\sqrt{n}E(f)} E(e^{\ic t S_n f/\sqrt{n}})$ converges to $e^{ -
(c -E(f)^2) t^2/2}$. Therefore, $e^{ - (c -E(f)^2) t^2/2}$ is
the characteristic function of a random variable $W$ and $(S_n
f-n E(f))/\sqrt{n}$ converges in distribution to $W$. This
yields $\sigma^2\coloneqq c-E(f)^2 \geq 0$, and
$W=\boN(0,\sigma^2)$, as desired. The proof of
Theorem~\ref{thm_main_general} is complete in this case.

We now turn to the other more interesting case, where $f\not\in
L^2$. We have apparently two different implications to prove,
but we will prove them at the same time, using the following
proposition.
\begin{prop}
\label{propEquiv} Let $T:X\to X$ and $\tilde T:\tilde X\to
\tilde X$ be two probability preserving mixing maps, and let
$f:X\to \R$ and $\tilde f:\tilde X\to \R$ be two functions with
the same distribution. Assume that both of them admit an
accurate characteristic expansion, and do not belong to $L^2$.
If $(\sum_{k=0}^{n-1}f\circ T^k-A_n)/B_n$ converges in
distribution to a nondegenerate random variable $W$, then
$(\sum_{k=0}^{n-1}\tilde f\circ \tilde T^k-A_n)/B_n$ also
converges to $W$.
\end{prop}
Let us show how this proposition implies
Theorem~\ref{thm_main_general}.
\begin{proof}[Conclusion of the proof of
Theorem~\ref{thm_main_general}, assuming
Proposition~\ref{propEquiv}] Let $f\not\in L^2$ admit an
accurate characteristic expansion.

Let $\tilde T$ be the left shift on the space $\tilde X=\R^\N$,
and let $\tilde f(x_0,x_1,\dots)=x_0$. We endow $\tilde X$ with
the product measure such that $\tilde f, \tilde f\circ \tilde
T,\dots$ are i.i.d.~and distributed as $f$. Then $\tilde f$
admits an accurate characteristic expansion (with $\tilde
\mu(t)=1$ and $\tilde \lambda(t)=E(e^{\ic tf})$).

Proposition~\ref{propEquiv} shows that the convergence of $(S_n
f-A_n)/B_n$ to $W$ gives the convergence of $(\sum_{k=0}^{n-1}
\tilde f\circ \tilde T^k -A_n)/B_n$ to $W$. This is one of the
desired implications in Theorem~\ref{thm_main_general}. The
other implication follows from the same argument, but
exchanging the roles of $T$ and $\tilde T$ in
Proposition~\ref{propEquiv}.
\end{proof}

The rest of this section is devoted to the proof of
Proposition~\ref{propEquiv}. We fix once and for all $T,\tilde
T$ and $f,\tilde f$ as in the assumptions of this proposition,
and assume that $(S_n f-A_n)/B_n$ converges in distribution to
a nondegenerate random variable $W$. Let us also fix $q$ and
$\epsilon$ such that $f\not\in L^q$ and $\lambda(t),\tilde
\lambda(t)$ satisfy \eqref{eqPasL2} (if the values of $q$ and
$\epsilon$ do not coincide for the expansions of $\lambda(t)$
and $\tilde\lambda(t)$, just take the minimum of the two).

Let $\Phi(t)=E(1-\cos(tf))\geq 0$, this function will play an
essential role in the following arguments.
\begin{lem}
\label{lem:heart}
We have $\int |e^{\ic tf}-1|^2=2 \Phi(t)$. Moreover, since $f$
does not belong to $L^2$,
  \begin{equation}
  \label{eq:ljkqmoiu}
  t^2+ \left(\int |e^{\ic tf}-1|\right)^2=o(\Phi(t))\quad \text{when }t\to 0.
  \end{equation}
\end{lem}
\begin{proof}
Writing $|e^{\ic tf}-1|^2=(e^{\ic tf}-1)(e^{-\ic tf}-1)$ and
expanding the product, the first assertion of the lemma is
trivial.

To prove that $t^2=o(\int |e^{\ic tf}-1|^2)$, let us show that
  \begin{equation}
  \label{eq:pasdepbtdeux}
  \int \left|\frac{e^{\ic tf}-1}{t}\right|^2 \dd m \to +\infty\quad \text{when
}t\to 0.
  \end{equation}
The integrand converges to $|f|^2$, whose integral is infinite.
Since a sequence $f_n$ of nonnegative functions always
satisfies $\int \liminf f_n \leq \liminf \int f_n$, by Fatou's
Lemma, we get \eqref{eq:pasdepbtdeux}.

Let us now check that $\left(\int |e^{\ic
tf}-1|\right)^2=o(\int |e^{\ic tf}-1|^2)$. Fix a large number
$M$ and partition the space into $A_M=\{|f|\leq M\}$ and
$B_M=\{|f|>M\}$. Since $(a+b)^2\leq 2a^2+2b^2$ for any $a,b\geq
0$, we get
  \begin{align*}
  \left(\int |e^{\ic tf}-1|\right)^2&
  =\left(\int_{A_M} |e^{\ic tf}-1|+\int_{B_M} |e^{\ic tf}-1|\right)^2
  \\&
  \leq 2\left(\int_{A_M} |e^{\ic tf}-1|\right)^2 + 2 \left(\int_{B_M}
|e^{\ic tf}-1|\right)^2
  \\&
  \leq 2M^2t^2 + 2 \norm{1_{B_M}}^2_{L^2} \norm{ e^{\ic tf}-1}^2_{L^2},
  \end{align*}
by Cauchy-Schwarz inequality. The term $2M^2t^2$ is negligible
with respect to $\int |e^{\ic tf}-1|^2$, by
\eqref{eq:pasdepbtdeux}, while the second term is $m(B_M)\int
|e^{\ic tf}-1|^2$. Choosing $M$ large enough, we can ensure
that $m(B_M)$ is arbitrarily small, concluding the proof.
\end{proof}

Lemma \ref{lem:heart} shows that \eqref{eqPasL2} is equivalent
to
  \begin{equation}
  \label{eqPasL2Bis}
  \lambda(t)=E(e^{\ic tf})+O(|t|^{q+\epsilon})+o(\Phi(t)).
  \end{equation}
The main difficulty is that, for a general function $f$ not
belonging to $L^q$, $|t|^{q+\epsilon}$ is not always negligible
with respect to $\Phi(t)$. This is however true along a
subsequence of $t$'s:

\begin{lem}
\label{lemDefLambda}
Since $f$ does not belong to $L^q$, there exists an infinite
set $A\subset \N$ such that, for any $t\in
\Lambda\coloneqq\bigcup_{n\in A} [2^{-n-1},2^{-n}]$,
  \begin{equation}
  \label{eqMinorePhi}
  |t|^{q+\epsilon/2}\leq \Phi(t).
  \end{equation}
\end{lem}
\begin{proof}
Assume by contradiction that, for any large enough $n$, there
exists $t_n\in [2^{-n-1},2^{-n}]$ with $\int_X 1-\cos(t_n f)<
|t_n|^{q+\epsilon/2}$. If $x\in X$ is such that $|f(x)|\in
[2^{n-1},2^n]$, then $|t_n f(x)|\in [1/4,1]$. Since $1-\cos(y)$
is bounded from below by $c>0$ on $[-1,-1/4]\cup[1/4,1]$, we
get
  \begin{equation*}
  m\{ |f|\in [2^{n-1},2^n]\} \leq c^{-1} \int 1-\cos(t_n f) \leq C
|t_n|^{q+\epsilon/2} \leq C 2^{-n(q+\epsilon/2)}.
  \end{equation*}
Hence, $\sum 2^{qn}m\{ |f|\in [2^{n-1},2^n]\}$ is finite. This
implies that $f$ belongs to $L^q$, a contradiction.
\end{proof}

\begin{lem}
\label{lemDecritLambda}
Along $\Lambda$, we have $|\lambda(t)|^2=1-(2+o(1)) \Phi(t)$.
\end{lem}
\begin{proof}
Along $\Lambda$, the previous lemma and \eqref{eqPasL2Bis} give
$\lambda(t)=E(e^{\ic tf})+o(\Phi(t))$. Hence,
  \begin{align*}
  |\lambda(t)|^2&=|E(e^{\ic tf})|^2+o(\Phi(t))
  \\&
  =(1 -E (1-e^{\ic tf}))\cdot (1-E(1-e^{-\ic tf})) + o(\Phi(t))
  \\&
  =1-2 E(1-\cos(tf)) + |E(1-e^{\ic tf})|^2+o(\Phi(t)).
  \end{align*}
Moreover, $E(1-\cos(tf))=\Phi(t)$, and $|E(1-e^{\ic tf})|^2
\leq \left( \int |1-e^{\ic tf}|\right)^2$, which is negligible
with respect to $\Phi(t)$ by Lemma~\ref{lem:heart}. This proves
the lemma.
\end{proof}

\begin{lem}
The sequence $B_n$ tends to infinity.
\end{lem}
\begin{proof}
Assume by contradiction that $B_n$ does not tend to infinity.
In this case, there exists a subsequence $j(n)$ such that the
distribution of $S_{j(n)} f-A_{j(n)}$ is tight. Since $T$ is
mixing, \cite[Theorem 2]{aaronson_weiss} implies the existence
of $c\in \R$ and of a measurable function $u:X\to \R$ such that
$f=u-u\circ T+c$ almost everywhere. In particular, $S_n f -nc$
converges in distribution, to $Z\coloneqq Z_1-Z_2$ where $Z_1$
and $Z_2$ are i.i.d.~and distributed as $u$. Hence, $e^{-\ic
tnc} E(e^{\ic tS_n f})$ converges to $E(e^{\ic tZ})$, and
therefore $|E(e^{\ic tS_n f})|\to |E(e^{\ic tZ})|$. However,
$E(e^{\ic tS_n f})=\mu(t)\lambda(t)^n+o(1)$. If
$|\lambda(t)|<1$, we obtain $E(e^{\ic tZ})=0$.

Along $\Lambda$, the function $\Phi$ is positive (by
\eqref{eqMinorePhi}) and $|\lambda(t)|^2=1-(2+o(1)) \Phi(t)$ by
the previous lemma. Hence, if $t$ is small enough and belongs
to $\Lambda$, we have $|\lambda(t)|<1$, and $E(e^{\ic tZ})=0$.
In particular, the function $t\mapsto E(e^{\ic tZ})$ is not
continuous at $0$, which is a contradiction since a
characteristic function is always continuous.
\end{proof}

\begin{lem}
The sequence $B_{n+1}/B_n$ converges to $1$.
\end{lem}
\begin{proof}
We know that $(S_n f-A_n)/B_n$ converges in distribution to a
nondegenerate random variable $W$. Since the measure is
invariant, $(S_n f\circ T -A_n)/B_n$ also converges to $W$.
Since $B_n\to\infty$, this implies that $(S_{n+1}f-A_n)/B_n$
converges to $W$. However, $(S_{n+1}f-A_{n+1})/B_{n+1}$
converges to $W$. The convergence of types theorem
(Theorem~\ref{thmConvTypes}) therefore yields $B_{n+1}/B_n\to
1$.
\end{proof}

Slowly varying functions have been defined in
Definition~\ref{defSlowlyVar}.

\begin{lem}
\label{lemDonned}
There exist $d\in (0,2]$ and a slowly varying function $L$ such
that $B_n\sim n^{1/d} L(n)$.
\end{lem}
\begin{proof}
Since $B_n\to \infty$, the convergence $(S_n f -A_n)/B_n \to W$
translates into: $e^{-\ic tA_n/B_n} \lambda(t/B_n)^n \to
E(e^{\ic tW})$ uniformly on small neighborhoods of $0$. Hence,
$|\lambda(t/B_n)|^{2n} \to |E(e^{\ic tW})|^2=E(e^{\ic tZ})$
where $Z\coloneqq W-W'$ is the difference of two independent
copies of $W$. Taking the logarithm, we get
  \begin{equation}
  \label{eqQSGqsg}
  2n\log |\lambda(t/B_n)|\to \log E(e^{\ic tZ}).
  \end{equation}
Since $B_n\to \infty$ and $B_{n+1}/B_n\to 1$, \cite[Proposition
1.9.4]{regular_variation} implies that, for $t>0$, one can
write $|\lambda(t)|^2=\exp(-t^d L_0(1/t))$ for some slowly
varying function $L_0$ and some real number $d$. Moreover,
$E(e^{\ic tZ})=e^{-ct^d}$ for some $c>0$. Since $E(e^{\ic tZ})$
is a characteristic function, this restricts the possible
values of $d$ to $d\in (0,2]$.

Let $t_0>0$ be such that $E(e^{\ic t_0Z})\in (0,1)$. The
convergence \eqref{eqQSGqsg} for $t=t_0$ becomes $n \sim C
B_n^d /L_0(B_n)$ for some $C>0$. Since $d>0$, the function
$x\mapsto C x^d /L_0(x)$ is asymptotically invertible by
\cite[Theorem 1.5.12]{regular_variation}, and admits an inverse
of the form $x\mapsto x^{1/d} L(x)$ where $L$ is slowly
varying. We get $B_n \sim n^{1/d}L(n)$.
\end{proof}

\begin{lem}
\label{lemEstimed}
The number $d$ given by Lemma~\ref{lemDonned} satisfies $d\leq
q+\epsilon/2$.
\end{lem}
\begin{proof}
Let $t_0>0$ satisfy $E(e^{\ic t_0Z})\in (0,1)$. The sequence
$|\lambda(t_0/B_n)|^{2n}$ converges to $E(e^{\ic t_0 Z})$.
Taking logarithms, we obtain the existence of $a>0$ such that
  \begin{equation}
  \label{eqSFOQSDFqsdf}
  -n \log \left|\lambda \left(\frac{t_0}{B_n} \right) \right|^2 \to a.
  \end{equation}

Since $B_{n+1}/B_n \to 1$, there exists $j(n)\to \infty$ such
that $t_0/B_{j(n)} \in \Lambda$ (where $\Lambda$ is defined in
Lemma~\ref{lemDefLambda}). Along this sequence, we have
$|\lambda(t_0/B_{j(n)})|^2=1-(2+o(1))\Phi(t_0/B_{j(n)})$ by
Lemma~\ref{lemDecritLambda}. Taking the logarithm and using
\eqref{eqSFOQSDFqsdf}, we obtain $j(n)\Phi(t_0/B_{j(n)}) \to
a/2$. By \eqref{eqMinorePhi}, this yields
$j(n)/B_{j(n)}^{q+\epsilon/2}=O(1)$. Moreover, by
Lemma~\ref{lemDonned} ,
  \begin{equation}
  j(n)/B_{j(n)}^{q+\epsilon/2}
  \sim j(n)^{1-(q+\epsilon/2)/d} /L(j(n))^{q+\epsilon/2}.
  \end{equation}
This sequence can be bounded only if $1- (q+\epsilon/2)/d \leq
0$, concluding the proof.
\end{proof}

\begin{proof}[Proof of Proposition~\ref{propEquiv}]
For small enough $t$, $|\lambda(t)-1|<1/2$. Hence, it is
possible to define $\log \lambda(t)$ by the series
$\log(1-s)=-\sum s^k/k$. Since the logarithm is a Lipschitz
function, \eqref{eqPasL2Bis} gives $\log \lambda(t)=\log
E(e^{\ic tf})+O(|t|^{q+\epsilon})+o(\Phi(t))$. Moreover,
$1-E(e^{\ic tf})=\Phi(t)- \ic E(\sin tf)$, hence
  \begin{equation}
  -\log E(e^{\ic tf})=\Phi(t) -\ic E(\sin tf)+o(\Phi(t))+O( |E(\sin tf)|^2).
  \end{equation}
Moreover, $|E(\sin(tf))|=|\Ima E(e^{\ic tf}-1)|\leq E|e^{\ic
tf}-1|$. Using \eqref{eq:ljkqmoiu}, we obtain
$|E(\sin(tf))|^2=o(\Phi(t))$. We have proved
  \begin{equation}
  \label{eqQSDFQSDF}
  -\log \lambda(t)=\Phi(t)-\ic E(\sin tf)+o(\Phi(t))+O(|t|^{q+\epsilon}).
  \end{equation}

The convergence $(S_n f-A_n)/B_n \to W$ also reads $e^{-\ic
tA_n/B_n} \lambda(t/B_n)^n \to E(e^{\ic tW})$. By
\eqref{eqQSDFQSDF}, the left hand side is
  \begin{equation*}
  \exp\Bigl(-\ic tA_n/B_n -n \Phi(t/B_n)+n\ic E(\sin(tf/B_n))+o(n\Phi(t/B_n))+O(n
/B_n^{q+\epsilon})\Bigr).
  \end{equation*}
By Lemma~\ref{lemEstimed}, $n/B_n^{q+\epsilon}$ tends to $0$
when $n\to \infty$. Hence, the last equation can also be
written as
  \begin{equation}
  \label{eqEstimeLambdan}
  \exp\Bigl(-\ic tA_n/B_n -n
\Phi(t/B_n)+n\ic E(\sin(tf/B_n))+o(n\Phi(t/B_n))+o(1))\Bigr).
  \end{equation}
To prove the desired convergence of $(\tilde S_n \tilde
f-A_n)/B_n$ to $W$, we should prove that $e^{-\ic tA_n/B_n}
\tilde \lambda(t/B_n)^n$ converges to $E(e^{\ic tW})$. The
previous arguments also apply to $\tilde \lambda$, and show
that
  \begin{multline}
  \label{eqTildeLambda}
  e^{-\ic tA_n/B_n} \tilde \lambda\left(\frac{t}{B_n} \right)^n=
  \\ \exp\Bigl(-\ic tA_n/B_n -n \Phi(t/B_n)+n\ic E(\sin(tf/B_n))+\tilde
o(n\Phi(t/B_n))+\tilde
  o(1))\Bigr),
  \end{multline}
where we have used the notation $\tilde o$ to emphasize the
fact that these negligible terms may be different from those in
\eqref{eqEstimeLambdan}.

Let us now conclude the proof by showing that
\eqref{eqTildeLambda} converges to $E(e^{\ic tW})$, using the
fact that \eqref{eqEstimeLambdan} converges to $E(e^{\ic tW})$.
The only possible problem comes from the negligible term
$\tilde o(n\Phi(t/B_n))$ (since the term $\tilde o(1)$ has no
influence on the limit).

We treat two cases. Assume first that $E(e^{\ic tW})\not=0$.
Then the modulus of $\lambda(t/B_n)^n$ converges to a nonzero
real number. In particular, $n\Phi(t/B_n)$ converges, which
implies that $\tilde o(n \Phi(t/B_n))$ converges to $0$. This
concludes the proof in this case.

Assume now that $E(e^{\ic tW})=0$. This implies that the
modulus of $\lambda(t/B_n)^n$ converges to $0$. By
\eqref{eqEstimeLambdan}, this yields $n\Phi(t/B_n)\to +\infty$.
In this case, we have no control on the argument of $e^{-\ic
tA_n/B_n}\tilde\lambda(t/B_n)^n$ (since the term $\tilde o(n
\Phi(t/B_n))$ may very well not tend to $0$), but its modulus
tends to $0$. This is sufficient to get again $e^{-\ic
tA_n/B_n}\tilde \lambda(t/B_n)^n \to 0=E(e^{\ic tW})$. This
concludes the proof of Proposition~\ref{propEquiv}.
\end{proof}


\section{Characteristic expansions for Gibbs-Markov maps}
\label{sec_proof_GM}

\subsection{The accurate characteristic expansion for
non-integrable functions}

Let us fix a mixing probability preserving Gibbs-Markov map
$T:X\to X$, as well as a measurable function $f:X\to \R$ with
$\sum m(a)Df(a)^\eta<\infty$ for some $\eta\in (0,1]$.

Let $\Transf$ denote the transfer operator associated to $T$
(defined by duality by $\int u\cdot v\circ T\dd m= \int \Transf
u \cdot v\dd m$). It is given explicitly by
  \begin{equation}
  \label{defLp}
  \Transf u(x)=\sum_{Ty=x} g(y)u(y),
  \end{equation}
where $g$ is the inverse of the jacobian of $T$. We will need
the following inequality: there exists a constant $C$ such that
  \begin{equation}
  \label{eq_bounds_jac}
  C^{-1}m(a)\leq g(x)\leq C m(a)
  \end{equation}
for any $a\in\alpha$ and $x\in a$. This follows from the
assumption of bounded distortion for Gibbs-Markov maps.

Let $\boL$ be the space of bounded functions $u:X\to \C$ such
that
  \begin{equation}
  \sup_{a\in\alpha} \sup_{x,y\in a} |u(x)-u(y)|/ d(x,y)^\eta<\infty.
  \end{equation}
Then $\Transf$ acts continuously on $\boL$, has a simple
eigenvalue at $1$ and the rest of its spectrum is contained in
a disk of radius $<1$. Moreover, it satisfies an inequality
  \begin{equation}
  \label{eqLY}
  \norm{\Transf^n u}_{\boL}\leq C\gamma^n
  \norm{u}_\boL+C\norm{u}_{L^1},
  \end{equation}
for some $\gamma<1$. This follows from \cite[Proposition 1.4
and Theorem 1.6]{aaronson_denker}.

Let us now define a perturbed transfer operator $\Transf_t$ by
$\Transf_t(u)=\Transf(e^{\ic t f} u)$. Using the estimate $\sum
m(a)Df(a)^\eta<\infty$, one can check that the operator
$\Transf_t$ acts continuously on $\boL$, and
  \begin{equation}
  \label{eq:EstimeClose}
  \norm{\Transf_t-\Transf}_{\boL\to\boL} =O(|t|^\eta+E|e^{\ic tf}-1|).
  \end{equation}
This follows from Lemma 3.5 and the proof of Corollary 3.6 in
\cite{gouezel_stable}.

The estimate~\eqref{eq:EstimeClose} is a strong continuity
estimate. We can therefore apply the following classical
perturbation theorem (which follows for instance from
\cite[Sections III.6.4 and IV.3.3]{kato_pe}).
\begin{thm}
\label{thm:kato} Let $A$ be a continuous operator on a Banach
space $\BB$, for which $1$ is a simple eigenvalue, and the rest
of its spectrum is contained in a disk of radius $<1$. Let
$A_t$ (for small enough $t$) be a family of continuous
operators on $\BB$, such that $\norm{A_t-A}_{\BB\to \BB} \to 0$
when $t\to 0$.

Then, for any small enough $t$, there exists a decomposition
$E_t\oplus F_t$ of $\BB$ into a one-dimensional subspace and a
closed hyperplane, such that $E_t$ and $F_t$ are invariant
under $A_t$. Moreover, $A_t$ is the multiplication by a scalar
$\lambda(t)$ on $E_t$, while $\norm{ (A_t)_{|F_t}^n}_{\BB\to
\BB} \leq C \gamma^n$ for some $\gamma<1$ and $C>0$.

The eigenvalue $\lambda(t)$ and the projection $P_t$ on $E_t$
with kernel $F_t$ satisfy
  \begin{equation}
  |\lambda(t)-1|\leq C \norm{A_t-A}_{\BB\to \BB}
  \end{equation}
and
  \begin{equation}
  \label{eqControleProjecteurFort}
  \norm{P_t-P_0}_{\BB\to \BB} \leq C \norm{A_t-A}_{\BB\to \BB}.
  \end{equation}
\end{thm}

This theorem yields an eigenvalue $\lambda(t)$ of $\Transf_t$
for small $t$, and an eigenfunction $\xi_t=P_t1 /\int P_t1$
such that $\int \xi_t=1$ and
  \begin{equation}
  \label{eqEstimexit}
  \norm{\xi_t-1}_{\boL}=O( |t|^\eta+ E|e^{\ic tf}-1|),
  \end{equation}
by \eqref{eqControleProjecteurFort}.

We have
  \begin{equation}
  \label{eqcarac}
  E(e^{\ic tS_nf})=\int \Transf_t^n(1) =\lambda(t)^n \int  P_t 1 + O(\gamma^n)
  =\mu(t)\lambda(t)^n+O(\gamma^n),
  \end{equation}
for $\mu(t)=\int P_t 1$. This proves that $f$ admits a
characteristic expansion. To prove Proposition~\ref{propGM}, we
have to show that this expansion is accurate, i.e., to get
precise estimates on $\lambda(t)$. We have
  \begin{equation}
  \lambda(t)=\int \lambda(t)\xi_t=\int \Transf_t \xi_t
  = \int\Transf_t 1 + \int (\Transf_t-\Transf)(\xi_t-1),
  \end{equation}
hence
  \begin{equation}
  \label{eqCentraleLambdat}
  \lambda(t)=E(e^{\ic tf})+ \int (e^{\ic tf}-1) (\xi_t-1).
  \end{equation}

When $\eta=1$ (i.e., $\sum m(a)Df(a)<\infty$) and $f\not\in
L^2$, \eqref{eqCentraleLambdat} together with the estimate
\eqref{eqEstimexit} readily imply that the characteristic
expansion of $f$ is accurate, concluding the proof of
Proposition~\ref{propGM} in this case. The general case
requires more work.

We first deal with the case $f\not\in L^{1+\eta/2}$. In this
case, we already have enough information to conclude:

\begin{lem}
\label{lempasL1} If $f\not\in L^{1+\eta/2}$, then $f$ admits an
accurate characteristic expansion.
\end{lem}
\begin{proof}
The equation \eqref{eqCentraleLambdat} together with
\eqref{eqEstimexit} yield
  \begin{equation}
  \label{eqisoiuqdfpoqs}
  \lambda(t)=E(e^{\ic tf})+ O\left(|t|^\eta\cdot \int |e^{\ic tf}-1|\right)+ O\left(
  \int |e^{\ic tf}-1|\right)^2.
  \end{equation}
Let $p\in [0,1]$ be such that $f\in L^p$ and $f\not\in
L^{p+\eta/2}$ (we use the convention that every measurable
function belongs to $L^0$). For any $x\in \R$, $|e^{\ic
x}-1|\leq 2 |x|^p$. Then
  \begin{equation}
  |t|^\eta \int |e^{\ic tf}-1|
  \leq 2 |t|^\eta \int |t|^p |f|^p
  \leq C |t|^{p+\eta}.
  \end{equation}
This yields the accurate characteristic expansion
\eqref{eqPasL2} as desired, for $q=p+\eta/2$ and
$\epsilon=\eta/2$.
\end{proof}

The case where $f\in L^{1+\eta/2}$ is a lot trickier. It
requires a more general spectral perturbation theorem,
essentially due to Keller and Liverani. Unfortunately, this
theorem is sufficient only when there exists $q<2$ such that
$f\not\in L^q$, while the remaining case can only be treated
using a generalization of this theorem, involving several
successive derivatives of the operators, that we will describe
in the next paragraph.

\subsection{A general spectral theorem}
\label{seckl}
In this paragraph, we describe a general spectral theorem
extending the results of \cite{keller_liverani} to the case of
several derivatives. A very similar result has been proved in
\cite{gouezel_liverani}, but with slightly stronger assumptions
that will not be satisfied in the forthcoming application to
Gibbs-Markov maps (in particular, \cite{gouezel_liverani}
requires \eqref{eq:Qbound} below to hold for $0\leq i<j\leq N$,
instead of $1\leq i<j\leq N$). Let us also mention
\cite{herve_pene} for related results.

Let $\BB_0\supset \BB_1\supset\dots \supset \BB_N$, $N\in
\N^*$, be a finite family of Banach spaces, let $I\subset \R$
be a fixed open interval containing $0$, and let
$\{\Lp_t\}_{t\in I}$ be a family of operators acting on each of
the above Banach spaces. Let also $b_0,b_1,\dots,b_{N-1}\in
(0,1]$ (usually, $b_i=1$ for $i\geq 1$). Let
$b(i,j)=\sum_{k=i}^{j-1}b_k$ for $0\leq i\leq j\leq N$. Assume
that
  \begin{equation}\label{eq:lypert1}
  \exists M>0, \forall\, t\in I,\quad
  \norm{\Lp_t^n f}_{\BB_0} \leq CM^n \norm{f}_{\BB_0}
  \end{equation}
and
  \begin{equation}\label{eq:lypert2}
  \exists\, \gamma<M,\;\;\forall\, t\in I,\quad
  \norm{\Lp_t^n f}_{\BB_1} \leq C \gamma^n \norm{f}_{\BB_1} +
  CM^n \norm{f}_{\BB_0}.
  \end{equation}
Assume also that there exist operators $Q_1,\ldots,Q_{N-1}$
satisfying the following properties:
  \begin{equation}\label{eq:Qbound}
  \forall\, 1\leq i <j \leq N,\quad
  \norm{Q_{j-i}}_{\BB_j\to \BB_{i}} \leq C
  \end{equation}
and, setting $\Delta_0(t):=\Lp_t$ and $\Delta_j(t):=\Lp_t
-\Lp_0 -\sum_{k=1}^{j-1} t^k Q_k$ for $j\geq 1$,
  \begin{equation}\label{eq:Cksmooth}
  \forall\, t\in I,\; \;
  \forall 0\leq i \leq j\leq N,
  \quad
  \norm{\Delta_{j-i}(t)}_{\BB_j \to \BB_{i}} \leq C|t|^{b(i,j)}.
  \end{equation}
These assumptions mean that $t\mapsto \Lp_t$ is continuous at
$t=0$ as a function from $\BB_i$ to $\BB_{i-1}$, and that
$t\mapsto \Lp_t$ even has a Taylor expansion of order $N-1$,
but the differentials take their values in weaker spaces.

For $\rad >\gamma$ and $\delta>0$, denote by $V_{\delta,\rad}$
the set of complex numbers $z$ such that $|z|\geq \rad$ and,
for all $1\leq k\leq N$, the distance from $z$ to the spectrum
of $\Lp_0$ acting on $\BB_k$ is $\geq \delta$.

\begin{thm}
\label{thm_1} Given a family of operators $\{\Lp_t\}_{t\in I}$
satisfying conditions \eqref{eq:lypert1}, \eqref{eq:lypert2},
\eqref{eq:Qbound} and \eqref{eq:Cksmooth} and setting
  \begin{equation*}
  R_N(t):= \sum_{k=0}^{N-1} t^k
  \sum_{\ell_1+\dots+\ell_j=k} (z-\Lp_0)^{-1} Q_{\ell_1} (z-\Lp_0)^{-1} \dots
  (z-\Lp_0)^{-1} Q_{\ell_j} (z-\Lp_0)^{-1},
  \end{equation*}
for all $z\in V_{\delta,\rad}$ and $t$ small enough, we have
  \begin{equation*}
  \norm{ (z-\Lp_t)^{-1} -R_N(t)}_{\BB_N \to \BB_0} \leq C
  |t|^{\kappa b_0+ b(1,N)}
  \end{equation*}
where $\kappa =\frac{\log(\rad/\gamma)}{\log(M/\gamma)}$.
\end{thm}
Hence, the resolvent $(z-\Lp_t)^{-1}$ depends on $t$ in a
$C^{\kappa b_0+ b(1,N)}$ way at $t=0$, when viewed as an
operator from $\BB_N$ to $\BB_0$.

Notice that one of the results of \cite{keller_liverani} in the
present setting reads
 \begin{equation}\label{eq:klest}
  \norm{ (z-\Lp_t)^{-1}- (z-\Lp_0)^{-1}}_{\BB_1 \to \BB_{0}} \leq C |t|^{\kappa b_0}.
  \end{equation}
Accordingly, one has Theorem \ref{thm_1} in the case $N=1$
where no assumption is made on the existence of the operators
$Q_j$.

We will use the following estimate of \cite{keller_liverani}:
\begin{lem}
\label{lemma} For any small enough $\tau$ and any $z\in
V_{\delta,\rad}$, we have
  \begin{equation}
  \norm{ (z-\Lp_0)^{-1}u}_{\BB_0}\leq C
\tau^{\kappa} \norm{u}_{\BB_1}+ C \tau^{\kappa-1} \norm{u}_{\BB_0}.
  \end{equation}
\end{lem}
\begin{proof}
This is essentially (11) in \cite{keller_liverani}. Let us
recall the proof for the convenience of the reader. We have
  \begin{equation}
  (z-\Lp_0)^{-1}=z^{-n} (z-\Lp_0)^{-1}\Lp_0^n +\frac{1}{z}\sum_{j=0}^{n-1}
  (z^{-1}\Lp_0)^j.
  \end{equation}
(this can be obtained for large enough $z$ by taking the series
expansion of $(z-\Lp_0)^{-1}$ and isolating the first terms).
Hence,
  \begin{align*}
  \norm{ (z-\Lp_0)^{-1} u }_{\BB_0}
  &\leq C|z|^{-n} \norm{(z-\Lp_0)^{-1}}_{\BB_1 \to \BB_1} \left[
  \gamma^n \norm{u}_{\BB_1} + M^n \norm{u}_{\BB_0} \right]
  \\&\ \ \ \
  + \frac{1}{|z|} \sum_{j=0}^{n-1} |z|^{-j} \norm{\Lp_0^j}_{\BB_0\to \BB_0}
\norm{u}_{\BB_0}
  \\&
  \leq C (\gamma/\rad)^n \norm{u}_{\BB_1} + C(M/\rad)^n \norm{u}_{\BB_0}.
  \end{align*}
Let us choose $n$ so that $(\gamma/\rad)^n =\tau^{\kappa}$,
i.e., $n=|\log \tau|/\log(M/\gamma)$. Then
  \begin{equation}
  (M/\rad)^n= \exp \left( |\log \tau| \cdot
\frac{\log(M/\rad)}{\log(M/\gamma)}\right)=\tau^{\kappa-1}.
  \qedhere
  \end{equation}
\end{proof}

\begin{proof}[Proof of Theorem \ref{thm_1}]
We have
  \begin{equation}
  (z-\Lp_t)^{-1}=(z-\Lp_0)^{-1}+ (z-\Lp_t)^{-1}(\Lp_t-\Lp_0)(z-\Lp_0)^{-1}.
  \end{equation}
If we want an expansion of $(z-\Lp_t)^{-1}$ up to order
$|t|^{\kappa b_0+b_1}$, this equation is sufficient: we can
replace on the right $(z-\Lp_t)^{-1}$ with $(z-\Lp_0)^{-1}$ up
to a small error $|t|^{\kappa b_0}$ (by \eqref{eq:klest}), and
use the Taylor expansion of $\Lp_t-\Lp_0$ to conclude (since
$\Lp_t-\Lp_0=O_{\BB_2 \to \BB_1}(|t|^{b_1})$, the global error
is of order $|t|^{\kappa b_0+b_1}$). If we want a better
precision $|t|^{\kappa b_0+b(1,N)}$, we should iterate the
previous equation, so that in the end $(z-\Lp_t)^{-1}$ is
multiplied by a term of order $|t|^{b(1,N)}$.

This is done as follows. Let
$A(z,t):=(\Lp_t-\Lp_0)(z-\Lp_0)^{-1}$. Iterating the previous
equation $N-1$ times, it follows that
\begin{equation}\label{eq:itera}
\begin{split}
  (z-\Lp_t)^{-1}&=\sum_{j=0}^{N-2}(z-\Lp_0)^{-1}
  A(z,t)^j+(z-\Lp_t)^{-1} A(z,t)^{N-1}\\
  &=\sum_{j=0}^{N-1}(z-\Lp_0)^{-1}
  A(z,t)^j+\left[(z-\Lp_t)^{-1}-(z-\Lp_0)^{-1}\right]A(z,t)^{N-1}.
\end{split}
\end{equation}
For each $j$, we then need to expand $A(z,t)^j$ to isolate the
good Taylor expansion, and negligible terms. The computation is
quite straightforward, but the notations are awful. To simplify
them, let us denote by $\uell$ a tuple $(\ell_1,\dots,\ell_k)$
of positive integers. Write also $l(\uell)=k$ and
$|\uell|=\ell_1+\dots+\ell_k$ and $\tilde
Q_\uell=Q_{\ell_1}(z-\Lp_0)^{-1}\dots Q_{\ell_k}
(z-\Lp_0)^{-1}$, and $\tilde \Delta_i(t)=\Delta_i(t)
(z-\Lp_0)^{-1}$.

Let us prove that, for any $j< N$,
  \begin{multline}\label{eqsldkfjmlqs}
  A(z,t)^j=\sum_{l(\uell)<j,\ j-l(\uell)<N-|\uell|}
  t^{|\uell|}  A(z,t)^{j-l(\uell)-1}\tilde
  \Delta_{N-|\uell|-(j-l(\uell)-1)}(t)\tilde Q_{\uell}
  \\
  +\sum_{l(\uell)=j,\ 0<N-|\uell|}t^{|\uell|}\tilde Q_{\uell}.
  \end{multline}

We start from the following equality, valid for each $j\in \N$
and $a\leq N$, which is a direct consequence of the definition
of $\Delta_a(t)$:
  \begin{equation}\label{eq:iteone}
  A(z,t)^j=A(z,t)^{j-1}\tilde \Delta_{a}(t)+\sum_{\ell=1}^{a-1}t^\ell
  A(z,t)^{j-1}\tilde Q_\ell.
  \end{equation}
We can again iterate this equation. We will adjust the
parameter $a$ we will use during this iteration, as follows: we
claim that, for all $1\leq m\leq j$,
  \begin{multline}\label{eq:iterm}
  A(z,t)^j=\sum_{l(\uell)<m,\ j-l(\uell)<N-|\uell|}
  t^{|\uell|} A(z,t)^{j-l(\uell)-1}\tilde
  \Delta_{N-|\uell|-(j-l(\uell)-1)}(t)\tilde Q_{\uell}\\
  +\sum_{l(\uell)=m,\ j-l(\uell)<N-|\uell|} t^{|\uell|}A(z,t)^{j-m}\tilde
Q_{\uell}\;.
  \end{multline}

In fact, for $m=1$ the above formula is just \eqref{eq:iteone}
for $a=N-j+1$. Next, suppose \eqref{eq:iterm} true for some
$m<j$, then the formula for $m+1$ follows by substituting the
last terms $A(z,t)^{j-m}$ using \eqref{eq:iteone} for
$a=N-|\uell|-(j-l(\uell)-1)$. This proves \eqref{eq:iterm} for
any $m\leq j$. In particular, for $m=j$, we obtain
\eqref{eqsldkfjmlqs}.

The equations \eqref{eq:itera} and \eqref{eqsldkfjmlqs} sum up
to
  \begin{multline}
  \label{eqRst}
  (z-\Lp_t)^{-1}=R_N(t) +\left[(z-\Lp_t)^{-1}-(z-\Lp_0)^{-1}\right]A(z,t)^{N-1}
  \\+
  \sum_{j=0}^{N-1} \sum_{l(\uell)<j,\ j-l(\uell)<N-|\uell|}
  t^{|\uell|}  (z-\Lp_0)^{-1}A(z,t)^{j-l(\uell)-1}\tilde
  \Delta_{N-|\uell|-(j-l(\uell)-1)}(t)\tilde Q_{\uell}.
  \end{multline}
We will show that all the error terms are $O_{\BB_N \to
\BB_0}(|t|^{\kappa b_0+b(1,N)})$.

Fix $j$ and $\uell$ with $l(\uell)<j,\ j-l(\uell)<N-|\uell|$.
Let
  \begin{equation}
  F(t)=t^{|\uell|} A(z,t)^{j-l(\uell)-1}\tilde
  \Delta_{N-|\uell|-(j-l(\uell)-1)}(t)\tilde Q_{\uell},
  \end{equation}
we want to show that
  \begin{equation}
  \label{eq_toproveF}
  \norm{(z-\Lp_0)^{-1} F(t)}_{\BB_N \to \BB_0} \leq C |t|^{\kappa b_0+b(1,N)}.
  \end{equation}
We have $\norm{|t|^{|\uell|} \tilde Q_{\uell}}_{\BB_N \to
\BB_{N-|\uell|}} \leq C |t|^{|\uell|} \leq C |t|^{
b(N-|\uell|,N)}$ by~\eqref{eq:Qbound}, while
  \begin{equation}
  \norm{\tilde
  \Delta_{N-|\uell|-(j-l(\uell)-1)}(t)}_{\BB_{N-|\uell|} \to
  \BB_{j-l(\uell)-1}} \leq C |t|^{b(j-l(\uell)-1, N-|\uell|)}
  \end{equation}
by~\eqref{eq:Cksmooth}, and
$\norm{A(z,t)^{j-l(\uell)-1}}_{\BB_{j-l(\uell)-1} \to \BB_0}
\leq C |t|^{b(0, j-l(\uell)-1)}$ again by~\eqref{eq:Cksmooth}
applied $j-l(\uell)-1$ times, since $A(z,t)=\tilde\Delta_1(t)$.
Multiplying these estimates gives
  \begin{equation}
  \label{eq_FtoB0}
  \norm{ F(t)}_{\BB_N \to \BB_0} \leq C |t|^{b(0,N)}.
  \end{equation}

Moreover, since $\tilde \Delta_k=\tilde\Delta_{k-1}-t^k \tilde
Q_k$, the norm of $\tilde \Delta_k$ from $\BB_j$ to
$\BB_{j-k+1}$ is bounded by $C|t|^{b(j-k+1,j)}$. In particular,
the norm of $\tilde \Delta_{N-|\uell|-(j-l(\uell)-1)}(t)$ from
$\BB_{N-|\uell|}$ to $\BB_{j-l(\uell)}$ is bounded by $C
|t|^{b(j-l(\uell), N-|\uell|)}$. Together with the same
arguments as above, we obtain
  \begin{equation}
  \label{eq_FtoB1}
  \norm{ F(t)}_{\BB_N \to \BB_1} \leq C |t|^{b(1,N)}.
  \end{equation}
The estimate \eqref{eq_toproveF} now follows from
\eqref{eq_FtoB0} and \eqref{eq_FtoB1}, as well as Lemma
\ref{lemma} for $\tau=|t|^{b_0}$.

We now turn to the term
  \begin{equation}
  \label{eqlastterm}
  \left[(z-\Lp_t)^{-1}-(z-\Lp_0)^{-1}\right]A(z,t)^{N-1}
  \end{equation}
of~\eqref{eqRst}. As $\norm{A(z,t)}_{B_i \to
\BB_{i-1}}=O(|t|^{b_{i-1}})$, we have
$\norm{A(z,t)^{N-1}}_{\BB_N\to \BB_1}=O(|t|^{b(1,N)})$. With
\eqref{eq:klest}, this shows that \eqref{eqlastterm} is
$O_{\BB_N \to \BB_0}(|t|^{\kappa b_0+b(1,N)})$, concluding the
proof.
\end{proof}

We will use the previous theorem in the following form:
\begin{cor}
\label{thm:kl2}
Under the assumptions of the previous theorem, assume also that
$M=1$ and that $\Lp_0$ acting on each space $\BB_j$ has a
simple isolated eigenvalue at $1$, with corresponding
eigenfunction $\xi_0$. Then, for small enough $t$, $\Lp_t$ has
a unique simple isolated eigenvalue $\lambda(t)$ close to $1$.

Let $\nu$ be a continuous linear form on $\BB_0$ with
$\nu(\xi_0)=1$. For small enough $t$, $\nu$ does not vanish on
the eigenfunction of $\Lp_t$ for the eigenvalue $\lambda(t)$.
It is therefore possible to define a normalized eigenfunction
$\xi_t$ satisfying $\nu(\xi_t)=1$.

Finally, there exist $u_1\in \BB_{N-1},\dots,u_{N-1} \in \BB_1$
such that, for any $\epsilon>0$,
  \begin{equation}
  \norm{\xi_t -\xi_0 -\sum_{k=1}^{N-1} t^k u_k }_{\BB_0} =O(
  |t|^{b(0,N)-\epsilon}).
  \end{equation}
\end{cor}
\begin{proof}
Let $c>0$ be small, and define an operator
$P_t=\frac{1}{2\ic\pi} \int_{|z-1|=c} (z-A_t)^{-1} \dd z$. The
operator $P_0$ is the spectral projection corresponding to the
eigenvalue $1$ of $P_0$. By Theorem~\ref{thm_1},
$\norm{P_t-P_0}_{\BB_N \to \BB_0}$ converges to $0$ when $t\to
0$. Therefore, the operator $P_t$ is also a rank one projection
for small enough $t$, corresponding to an eigenvalue
$\lambda(t)$ of $A_t$. Let $\tilde \xi_t=P_t(u)$ for some fixed
$u\in \BB_N$ with $P_0(u)\not=0$, $\tilde\xi_t$ is an
eigenfunction of $A_t$ for the eigenvalue $\lambda(t)$. Since
$\norm{\tilde \xi_t-\tilde\xi_0}_{\BB_0} \to 0$, this
eigenfunction satisfies $\nu(\tilde\xi_t)\not=0$ for small
enough $t$, and we can define a normalized eigenfunction
$\xi_t=\tilde \xi_t/\nu(\tilde \xi_t)$.

For $1\leq k\leq N-1$, let
  \begin{equation*}
  \tilde u_k=\sum_{\ell_1+\dots+\ell_j=k}\frac{1}{2\ic\pi}\int_{|z-1|=c}
  (z-A_0)^{-1} Q_{\ell_1}\dots (z-A_0)^{-1} Q_{\ell_j}
  (z-A_0)^{-1}u \dd z.
  \end{equation*}
It belongs to $\BB_{N-k}$ by \eqref{eq:Qbound}. Moreover,
Theorem~\ref{thm_1} yields
  \begin{equation}
  \label{eq_exptilde}
  \norm{\tilde\xi_t -\tilde \xi_0 -\sum_{k=1}^{N-1} t^k \tilde u_k}_{\BB_0}
  \leq C|t|^{\kappa b_0+b(1,N)},
  \end{equation}
for $\kappa=\log( (1-c)/\gamma)/\log(1/\gamma)$. Applying $\nu$
to this equation, we obtain that $\nu(\tilde \xi_t)$ admits an
expansion $\nu(\tilde\xi_t)=\sum_{k=0}^{N-1} t^k \nu_k+O(
|t|^{\kappa b_0+b(1,N)})$. Hence, $\xi_t=\tilde
\xi_t/\nu(\tilde \xi_t)$ also admits an expansion similar to
\eqref{eq_exptilde}.

This is almost the conclusion of the proof, we only have to see
that the error term $O( |t|^{\kappa b_0+b(1,N)})$ can be
modified to be of the form $O(|t|^{b(0,N)-\epsilon})$ for any
$\epsilon>0$. This follows from the fact that $c$ can be chosen
arbitrarily small (by holomorphy, this does not change the
projection $P_t$ for small enough $t$, hence $\tilde u_k$ and
$u_k$ are also not modified).
\end{proof}

\begin{rmk}
Corollary \ref{thm:kl2} states that the normalized
eigenfunction $\xi_t$ has a Taylor expansion of order
$b(0,N)-\epsilon$ at $0$. Under similar assumptions at every
point of a neighborhood $I$ of $0$, we obtain that $\xi_t$ has
a Taylor expansion at every point of $I$. By a lemma of
Campanato \cite{campanato}, this implies that $\xi_t$ is
$C^{b(0,N)-\epsilon}$ on $I$, a result analogous to
\cite{herve_pene}.
\end{rmk}

\subsection{Definition of good Banach spaces}
We now turn back to the dynamical setting: $T:X\to X$ is a
mixing Gibbs-Markov map, and $f:X\to \R$ is a function
satisfying $\sum m(a)Df(a)^\eta<\infty$, for which we want to
prove an accurate characteristic expansion. To do that, we wish
to apply Corollary~\ref{thm:kl2} to a carefully chosen sequence
of Banach space. We have currently at our disposal the spaces
$L^p$ (but the spectral properties of the transfer operator on
these spaces are not good), and the space $\boL$ (which is only
a space, not a sequence of spaces). Our goal in this paragraph
is to define a family of intermediate spaces between $L^p$ and
$\boL$, which will be more suitable to apply
Corollary~\ref{thm:kl2}.

For $1\leq p\leq \infty$ and $\reg>0$, let us define a Banach
space $\boL^{p,\reg}$ as follows: it is the space of measurable
functions $u$ such that, for any $k\in \N$, there exists a
decomposition $u=v+w$ with $\norm{v}_{\boL} \leq C e^k$ and
$\norm{w}_{L^p} \leq C e^{-\reg k}$. The best such $C$ is by
definition the norm of $u$ in $\boL^{p,\reg}$. This Banach
space is an \emph{interpolation space} between $\boL$ and $L^p$
(see \cite{bergh_lofstrom_interpolation}).

Of course, $\boL^{p,\reg}$ is included in $L^p$ (simply use the
decomposition for $k=0$), and $\boL^{p,\reg}$ is contained in
$\boL^{p',\reg'}$ when $p'\leq p$ and $\reg'\leq \reg$.

Let us check that the operators $\Transf$ and $\Transf_t$ enjoy
good spectral properties when acting on $\boL^{p,\reg}$. This
will be a consequence of the fact that they have good
properties when acting on $\boL$, and are contractions when
acting on $L^p$.

\begin{lem}
\label{lemly1} Let $1\leq p\leq \infty$ and let $\reg>0$. The
operator $\Transf$ acts continuously on the space
$\boL^{p,\reg}$. Moreover, there exist $\gamma<1$ and $C>0$
such that
  \begin{equation}
  \label{eq_LY_Bpa}
  \norm{ \Transf^n u}_{\boL^{p,\reg}} \leq C \gamma^n \norm{u} _{\boL^{p,\reg}}
  +C \norm{u}_{L^1}.
  \end{equation}
\end{lem}
\begin{proof}
Let $\gamma_0<1$ be such that $\norm{\Transf^n u}_{\boL} \leq C
\gamma_0^n \norm{u}_{\boL} + C\norm{u}_{L^1}$.

For $n\in \N$, let $A$ be the integer part of $\epsilon n$, for
some $\epsilon>0$ with $\gamma_0 e^{-\epsilon}<1$. Let $u\in
\boL^{p,\reg}$, there exists a decomposition $u=v+w$ with
$\norm{v}_{\boL}\leq e^{k+A}\norm{u}_{\boL^{p,\reg}}$ and
$\norm{w}_{L^p} \leq e^{-\reg(k+A)}\norm{u}_{\boL^{p,\reg}}$.
Then
  \begin{align*}
  \norm{\Transf^n v}_{\boL} & \leq C \gamma_0^n e^{k+A}\norm{u}_{\boL^{p,\reg}} + C
  \norm{v}_{L^1}
  \\&
  \leq C (\gamma_0^n e^A) e^k \norm{u}_{\boL^{p,\reg}} + C
\norm{w}_{L^1}+C \norm{u}_{L^1}
  \\&\leq C (\gamma_0^n e^A + e^{-\reg(k+A)}) e^k \norm{u}_{\boL^{p,\reg}} + C
\norm{u}_{L^1}
  \\& \leq e^k (C \gamma^n\norm{u}_{\boL^{p,\reg}}+C \norm{u}_{L^1}),
  \end{align*}
for some $\gamma<1$. Moreover
  \begin{align*}
  \norm{\Transf^n w}_{L^p}&\leq \norm{w}_{L^p} \leq C e^{-\reg A}
  e^{-\reg k}\norm{u}_{\boL^{p,\reg}}
  \leq e^{-\reg k} (C \gamma^n\norm{u}_{\boL^{p,\reg}})
  \\&\leq e^{-\reg k} (C \gamma^n\norm{u}_{\boL^{p,\reg}}+C \norm{u}_{L^1})
  \end{align*}
for some $\gamma<1$.

Therefore, the decomposition of $\Transf^n u$ as $\Transf^n
v+\Transf^n w$ shows that $\Transf^n u$ belongs to
$\boL^{p,\reg}$, and has a norm bounded by $C
\gamma^n\norm{u}_{\boL^{p,\reg}}+C \norm{u}_{L^1}$.
\end{proof}

\begin{lem}
For any $p\geq 1$ and $\reg>0$, the inclusion of
$\boL^{p,\reg}$ in $L^1$ is compact.
\end{lem}
\begin{proof}
Let $u_n$ be a sequence bounded by $1$ in $\boL^{p,\reg}$. Fix
$k\in \N$, and let us decompose $u_n$ as $v_n+w_n$ with
$\norm{v_n}_{\boL} \leq e^k$ and $\norm{w_n}_{L^p}\leq e^{-\reg
k}$. Since the inclusion of $\boL$ in $L^1$ is compact, there
exists a subsequence $j(n)$ such that $v_{j(n)}$ converges in
$L^1$. Therefore, $\limsup_{n,m\to \infty}
\norm{u_{j(n)}-u_{j(m)}}_{L^1} \leq 2 e^{-\reg k}$. With a
diagonal argument over $k$, we finally obtain a convergent
subsequence of $u_n$.
\end{proof}

\begin{cor}
The transfer operator $\Transf$ acting on $\boL^{p,s}$ has a
simple eigenvalue at $1$, and the rest of its spectrum is
contained in a disk of radius $<1$.
\end{cor}
\begin{proof}
Together with Hennion's Theorem \cite{hennion}, the two
previous lemmas ensure that the essential spectral radius of
$\Transf$ acting on $\boL^{p,s}$ is $\leq \gamma<1$, i.e., the
elements of the spectrum of $\Transf$ with modulus $>\gamma$
are isolated eigenvalues of finite multiplicity.

If $u$ is an eigenfunction of $\Transf$ for an eigenvalue of
modulus $1$, then $u$ belongs to $L^1$. Since $\Transf$
satisfies a Lasota-Yorke inequality \eqref{eqLY} on the space
$\boL$, the theorem of Ionescu-Tulcea and Marinescu
\cite{ionescu_tulcea_marinescu} implies that $u$ belongs to
$\boL$. However, we know that $\Transf$ acting on $\boL$ has a
simple eigenvalue at $1$, and no other eigenvalue of modulus
$1$.
\end{proof}

\begin{lem}
\label{lemly2} For any $p\geq 1$ and $\reg>0$, the operator
$\Transf_t$ acts continuously on $\boL^{p,\reg}$ for small
enough $t$. Moreover, $\norm{\Transf_t
-\Transf}_{\boL^{p,\reg}\to\boL^{p,\reg}}$ converges to $0$
when $t\to 0$. Finally, if $t$ is small enough, $\Transf_t$
satisfies a Lasota-Yorke type inequality
  \begin{equation}
  \label{LYt}
  \norm{\Transf_t^n u}_{\boL^{p,\reg}} \leq C\gamma^n \norm{u}_{\boL^{p,\reg}} + C \norm{u}_{L^1},
  \end{equation}
where $C>0$ and $\gamma<1$ are independent of $t$.
\end{lem}
\begin{proof}
For any operator $M$ sending $\boL$ to $\boL$ and $L^p$ to
$L^p$, then $M$ sends $\boL^{p,\reg}$ to $\boL^{p,\reg}$ and,
for any integer $A\geq 0$,
  \begin{equation}
  \label{eq_interborn}
  \norm{M}_{\boL^{p,\reg}\to\boL^{p,\reg}} \leq \max(e^{A} \norm{M}_{\boL \to \boL},
e^{-\reg A} \norm{M}_{L^p \to
  L^p}).
  \end{equation}
This follows using the decomposition of $u\in \boL^{p,\reg}$ as
$v+w$ with $\norm{v}_{\boL}\leq e^{k+A}
\norm{u}_{\boL^{p,\reg}}$ and $\norm{w}_{L^p}\leq e^{-\reg
k-\reg A} \norm{u}_{\boL^{p,\reg}}$.

By~\eqref{eq:EstimeClose}, $\norm{\Transf_t -\Transf}_{\boL\to
\boL}$ tends to $0$, while $\norm{\Transf_t -\Transf}_{L^p\to
L^p}$ is uniformly bounded. Applying \eqref{eq_interborn} to
$M=\Transf_t-\Transf$ and $e^A$ close to
$\norm{\Transf_t-\Transf}_{\boL \to \boL}^{-1/2}$, we obtain
that $\norm{\Transf_t -\Transf}_{\boL^{p,\reg}\to
\boL^{p,\reg}}$ tends to zero.

By \eqref{eq_LY_Bpa}, we can fix $N>0$, $\sigma<1$ and $C>0$
such that $\norm{ \Transf^N u}_{\boL^{p,\reg}} \leq \sigma
\norm{u}_{\boL^{p,\reg}} +C \norm{u}_{L^1}$. Let $\sigma_1\in
(\sigma,1)$. Since $\norm{\Transf_t -\Transf}_{\boL^{p,\reg}\to
\boL^{p,\reg}}$ tends to $0$ when $t\to 0$, the previous
equation gives, for small enough $t$,
  \begin{equation}
  \norm{ \Transf_t^N u}_{\boL^{p,\reg}} \leq \sigma_1 \norm{u}_{\boL^{p,\reg}} +C \norm{u}_{L^1}
  \end{equation}
Iterating this equation, we get by induction over $k$
  \begin{equation}
  \norm{ \Transf_t^{kN} u}_{\boL^{p,\reg}} \leq \sigma_1^k \norm{u} _{\boL^{p,\reg}} +
  C\sum_{j=0}^{k-1} \sigma_1^{k-1-j} \norm{\Transf_t^{jN} u}_{L^1}.
  \end{equation}
Since $\Transf_t$ is a contraction on $L^1$, we obtain $\norm{
\Transf_t^{kN} u}_{\boL^{p,\reg}} \leq \sigma_1^k \norm{u}
_{\boL^{p,\reg}} + C'\norm{u}_{L^1}$, for $C'=C
\sum_{j=0}^\infty \sigma^j$. This proves \eqref{LYt} for $n$ of
the form $kN$, and the general case follows.
\end{proof}

\subsection{Gaining \texorpdfstring{$\delta$}{d} in the integrability exponent}

We wish to apply Corollary \ref{thm:kl2} to obtain the accurate
characteristic expansion. This theorem involves an
(arbitrarily) small loss of $\epsilon$, that we will have to
compensate at some point. In this paragraph, we show how a
regularity assumption of the form $\sum m(a)Df(a)^\eta<\infty$
makes it possible to obtain a definite gain in the
integrability exponent of some functions, which ultimately will
compensate the aforementioned loss.

\begin{lem}
\label{gain_integr} For any $\beta\in (0,1]$, there exists
$\delta>0$ with the following property. Let $f\in L^p$ (for
some $p \in [1, 1/\beta]$) satisfy $\sum m(a)
Df(a)^\beta<\infty$. Let $c\in [\beta,p]$, and consider a
function $u$ such that $|u|\leq |f|^c$, and, for all $a\in
\alpha$,
  \begin{equation}
  Du(a) \leq \begin{cases}
  Df(a) & \text{ if }c\leq 1,\\
  Df(a) \norm{1_a f}_{L^\infty}^{c-1} & \text{ if }c>1.
  \end{cases}
  \end{equation}
Let $q$, $r$ be positive numbers (possibly $q=\infty$) such
that $1/r= 1/(p/c) + 1/q$, and $r\geq 1+\beta$. Then the
operator $v\mapsto \Transf (u v)$ maps $L^q$ to $L^{r+\delta}$
(and its norm is bounded only in terms of $f$ and $\beta$).
\end{lem}
Since $u \in L^{p/c}$, the H\"{o}lder inequality shows that the
operator $v\mapsto \Transf (u v)$ maps $L^q$ to $L^r$. The
lemma asserts that there is in fact a small gain of $\delta$ in
the integrability exponent, due to the regularity property
$\sum m(a) Df(a)^\beta<\infty$. Moreover, the gain is uniform
over the parameters if $r$ stays away from $1$.
\begin{proof}
We will show that, under the assumptions of the lemma, the
operator $v\mapsto \Transf(uv)$ maps $L^q$ to $L^{\tilde r}$,
for $\tilde r= \frac{pq/c -\beta^2 q}{p/c+q-\beta^2 q}$. Since
$\tilde r -r$ is uniformly bounded from below when the
parameters vary according to the conditions of the lemma, this
will conclude the proof.\footnote{Indeed,
  \begin{equation}
  \label{eqtilderr}
  \tilde r -r =\frac{\beta^2 q \left( pq/c-p/c-q \right)}{(p/c+q)(p/c+q-\beta^2 q)}.
  \end{equation}
Since $r\geq 1+\beta$, we have $pq/c \geq (1+\beta)(p/c+q)$.
Therefore, the second term of the numerator of
\eqref{eqtilderr} is at least $\beta (p/c +q)$. Simplifying
with the denominator, we get
  \begin{equation}
  \tilde r-r \geq \frac{\beta^3}{ p/qc + 1 -\beta^2} \geq \frac{\beta^3}{\beta^{-2}+1-\beta^2}.
  \end{equation}
}

Let us first show that
  \begin{equation}
  \label{sumfbounded}
  \sum_{a\in \alpha} m(a) \norm{f 1_a}_{L^\infty}^\beta <\infty.
  \end{equation}
For $x,y\in a$, we have $|f(x)|\leq |f(y)|+Df(a)$. Integrating
over $y$, we get $|f(x)|\leq \frac{1}{m(a)} \int_a |f| +
Df(a)$. Together with the inequality $(t'+t)^\beta\leq 1+t' +
t^\beta$, valid for any $t',t\geq 0$, we obtain
  \begin{equation*}
  \sum_{a\in\alpha} m(a) \norm{f 1_a}_{L^\infty}^\beta
  \leq \sum_{a\in\alpha} m(a) \left( 1+\frac{1}{m(a)}\int_a |f|\right)
  + \sum_{a\in\alpha} m(a) Df(a)^\beta.
  \end{equation*}
These sums are finite, concluding the proof of
\eqref{sumfbounded}.

Let $\tilde\beta=\beta^2$, we will now show that
  \begin{equation}
  \label{sumubounded}
  \sum_{a\in \alpha} m(a) \norm{u 1_a}_{L^\infty}^{\tilde\beta} <\infty.
  \end{equation}
By the previous argument, it is sufficient to show that
$\sum_{a\in \alpha} m(a)Du(a)^{\tilde\beta}$ is finite. If
$c\leq 1$, $Du(a)^{\tilde\beta}\leq Df(a)^{\tilde\beta} \leq
\max(1,Df(a)^\beta)$, and the result follows. If $c>1$,
  \begin{align*}
  Du(a)^{\tilde\beta} &\leq Df(a)^{\tilde\beta} \norm{1_a f}_{L^\infty}^{(c-1){\tilde\beta}}
  \leq \max(Df(a),\norm{1_a f}_{L^\infty})^{c{\tilde\beta}}
  \\& \leq Df(a)^{c{\tilde\beta}} + \norm{1_a f}_{L^\infty}^{c{\tilde\beta}}.
  \end{align*}
Since $c{\tilde\beta}\leq \beta$, \eqref{sumfbounded} shows
that $\sum m(a)Du(a)^{\tilde\beta}<\infty$, concluding the
proof of \eqref{sumubounded}.

By \eqref{defLp} and \eqref{eq_bounds_jac},
$\Transf(|u|^{\tilde\beta})$ is bounded by $\sum_{a\in\alpha}
m(a) \norm{u 1_a}_{L^\infty}^{\tilde\beta}$, which is finite.
Hence, $\Transf(|u|^{\tilde\beta})$ is a bounded function.

Let us now estimate $\Transf(uv)$ for $v\in L^q$. Let
$\rho=\tilde r/(\tilde r-1)$, so that $1/\rho+1/\tilde r=1$. We
have
  \begin{equation}
  \Transf(|uv|) = \Transf(|u|^{{\tilde\beta}/\rho} |u|^{1-{\tilde\beta}/\rho} |v|) \leq \Transf(
  |u|^{{\tilde\beta}})^{1/\rho} \Transf( |u|^{\tilde
  r(1-{\tilde\beta}/\rho)} |v|^{\tilde
  r}) ^{1/\tilde r}.
  \end{equation}
Since $\Transf( |u|^{{\tilde\beta}})$ is bounded, we obtain
  \begin{equation*}
  \norm{\Transf(uv)}_{L^{\tilde r}} \leq C \left( \int \Transf
  (|u|^{\tilde r(1-{\tilde\beta}/\rho)} |v|^{\tilde r})
  \right)^{1/\tilde r}
  =C \left(\int |u|^{\tilde r(1-{\tilde\beta}/\rho)} |v|^{\tilde r} \right)^{1/\tilde r}.
  \end{equation*}
Let $s$ and $t$ be such that $1/s+1/t=1$ and $t\tilde r=q$,
i.e., $t=q/\tilde r$ and $s=q/(q-\tilde r)$. The H\"{o}lder
inequality gives
  \begin{equation}
  \int |u|^{\tilde r(1-{\tilde\beta}/\rho)} |v|^{\tilde r}
  \leq \left( \int |u|^{\tilde r(1-{\tilde\beta}/\rho) s}\right)^{1/s} \left( \int |v|^{\tilde rt}
  \right)^{1/t}.
  \end{equation}
The choice of $\tilde r$ above ensures that $\tilde
r(1-{\tilde\beta}/\rho)s=p/c$. Hence, the integral involving
$u$ is finite, since $u\in L^{p/c}$. We obtain
$\norm{\Transf(uv)}_{L^{\tilde r}} \leq C \norm{v}_{L^q}$, as
required.
\end{proof}

\begin{lem}
\label{gain_integrfort} For any $\beta\in
(0,1]$, there exists $\delta>0$ with the following property.
Let $f\in L^p$ (for some $p \in [1, 1/\beta]$) satisfy $\sum
m(a) Df(a)^\beta<\infty$. Let $c\in [\beta,p]$, and consider a
function $u$ such that $|u|\leq |f|^c$, and, for all $a\in
\alpha$,
  \begin{equation}
  Du(a) \leq \begin{cases}
  Df(a) & \text{ if }c\leq 1,\\
  Df(a) \norm{1_a f}_{L^\infty}^{c-1} & \text{ if }c>1.
  \end{cases}
  \end{equation}
Let $q$, $r$ be positive numbers (possibly $q=\infty$) such
that $1/r= 1/(p/c) + 1/q$, and $r\geq 1+\beta$. Then, for any
$\reg>0$, there exists $\reg'=\reg'(f,\beta,\reg)$ such that
the operator $v\mapsto \Transf (u v)$ maps $\boL^{q, \reg}$ to
$\boL^{r+\delta, \reg'}$ (and its norm is bounded only in terms
of $f,\beta,s$).
\end{lem}
\begin{proof}
Let $\delta_0$ be the value of $\delta$ given by Lemma
\ref{gain_integr} for $\beta/2$ instead of $\beta$. We will
prove that the lemma holds for $\delta=\delta_0/2$.

For $K\geq 1$, denote by $A(K)$ the union of the elements
$a\in\alpha$ such that $Df(a)+\norm{1_a f}_{L^\infty} \leq K$,
and let $B(K)$ be its complement. The finiteness of the sum
$\sum_{a\in \alpha} m(a) (Df(a)^\beta+ \norm{1_a
f}_{L^\infty}^\beta)$ (which has been proved in
\eqref{sumfbounded}) implies that there exists $C$ such that
  \begin{equation}
  \label{eqbornemesureBK}
  m(B(K)) \leq C K^{-\beta}.
  \end{equation}
Moreover, let $d=\max(c,1)$, then $u$ is bounded by $K^d$ on
$A(K)$, and its Lipschitz constant is also bounded by $K^d$.
Therefore,
  \begin{equation}
  \label{eqbornedansboL}
  \norm{ 1_{A(K)} u}_{\boL} \leq C K^d.
  \end{equation}

Take $v\in \boL^{q,\reg}$ bounded by $1$, and $k\in \N$. By
definition of $\boL^{q,\reg}$, we can write $v=w+w'$ with
$\norm{w}_{\boL} \leq e^k$ and $\norm{w'}_{L^q}\leq e^{-\reg
k}$. For any $K\geq 1$, we obtain a decomposition of
$\Transf(uv)$ as the sum of $\Transf( 1_{A(K)} u w)$ and
$\Transf( 1_{B(K)} u w+ uw')$. We claim that
  \begin{equation}
  \label{eqdansboL}
  \norm{\Transf( 1_{A(K)} u w)}_{\boL} \leq C K^d e^{k}
  \end{equation}
and, for some $\epsilon>0$,
  \begin{equation}
  \label{eqdansLr}
  \norm{\Transf( 1_{B(K)} u w+uw')}_{L^{r+\delta_0/2}} \leq C e^{-\reg k} + C e^k K^{-\epsilon}.
  \end{equation}
This concludes the proof of the lemma, for $\delta=\delta_0/2$
and $\reg'=\reg/( 1+ d(1+\reg)/\epsilon)$. Indeed, for $K=\exp(
(1+\reg)k/\epsilon)$, the bound in \eqref{eqdansboL} becomes $C
e^{\reg k/\reg'}$, and the bound in \eqref{eqdansLr} becomes
$Ce^{-\reg k}$. Taking $k'$ close to $\reg k/\reg'$, we have
obtained a decomposition of $\Transf( uv)$ as a sum $\tilde
w+\tilde w'$ with $\norm{\tilde w}_{\boL} \leq C e^{k'}$ and
$\norm{\tilde w'}_{L^{r+\delta_0/2}} \leq C e^{-\reg' k'}$, as
desired.

It remains to prove \eqref{eqdansboL} and \eqref{eqdansLr}. The
former follows from the inequality $\norm{zz'}_{\boL}\leq C
\norm{z}_{\boL} \norm{z'}_\boL$, applied to the functions
$z=1_{A(K)}u$ (whose norm is bounded by
\eqref{eqbornedansboL}), and $z'=w$ (whose norm is at most
$e^k$).

We turn to \eqref{eqdansLr}. First, by Lemma \ref{gain_integr},
  \begin{equation}
  \norm{\Transf(uw')}_{L^{r+\delta_0/2}}\leq \norm{\Transf(uw')}_{L^{r+\delta_0}}
  \leq C \norm{w'}_{L^q},
  \end{equation}
which is at most $e^{-\reg k}$. Let then $Q$ be large enough,
and let $r'$ be such that $1/r'=1/(p/c)+1/Q$. Then
  \begin{equation}
  r-r'=rr' \left(\frac{1}{r'}-\frac{1}{r} \right)
  =rr' \left( \frac{1}{Q} -\frac{1}{q}\right)
  \leq \frac{rr'}{Q}.
  \end{equation}
Moreover, $1/r \geq 1/(p/c) \geq \beta^2$, and $1/r'\geq
\beta^2$ as well. Hence, $r-r' \leq \beta^{-4}/Q$. Choosing $Q$
large enough, we can ensure $r-r'\leq \delta_0/2$. Therefore,
  \begin{align*}
  \norm{\Transf( 1_{B(K)} u w)}_{L^{r+\delta_0/2}}
  &\leq \norm{w}_{L^\infty} \norm{\Transf( 1_{B(K)} |u| )}_{L^{r+\delta_0/2}}
  \\&\leq \norm{w}_{L^\infty} \norm{\Transf( 1_{B(K)} |u| )}_{L^{r'+\delta_0}}.
  \end{align*}
Since $1/(p/c) \leq 1/r \leq 1/(1+\beta)$, we have $1/r'\leq
1/(1+\beta)+1/Q$, which is at most $1/(1+\beta/2)$ if $Q$ is
large enough. Thanks to the definition of $\delta_0$ above, we
can therefore apply Lemma \ref{gain_integr} to the function
$v=1_{B(K)}$ and the parameters $r',Q,\beta/2$, to obtain
  \begin{equation}
  \norm{\Transf( 1_{B(K)} |u| )}_{L^{r'+\delta_0}}
  \leq C \norm{ 1_{B(K)}}_{L^Q}.
  \end{equation}
Since $\norm{w}_{L^\infty}\leq e^k$ and  $\norm{
1_{B(K)}}_{L^Q} \leq C K^{-\beta/Q}$ by
\eqref{eqbornemesureBK}, this proves \eqref{eqdansLr} for
$\epsilon=\beta/Q$.
\end{proof}

\subsection{Accurate characteristic expansions for integrable functions}

We will now prove that a function $f$ satisfying $\sum
m(a)Df(a)^\eta<\infty$ admits an admissible characteristic
expansion. By Lemma \ref{lempasL1}, it is sufficient to treat
the case $f\in L^{1+\eta/2}$. We will give very precise
asymptotics of the eigenvalue $\lambda(t)$ of the transfer
operator, yielding also other limit theorems in the $L^2$ case.

\begin{thm}
\label{thm_highdiff} Let $\eta\in (0,1]$. There exists a
function $\epsilon: (1,\infty) \to \R_+^*$, bounded away from
zero on compact subsets of $(1,\infty)$, with the following
property.

Let $f$ satisfy $\sum m(a)Df(a)^\eta<\infty$, and $f\in L^p$
for some $p>1$. Then there exist complex numbers $c_i$ (for
$1\leq i < p+\epsilon(p)$) such that
  \begin{equation}
  \label{eqdevelopthighdegree}
  \lambda(t)=E(e^{\ic tf})+ \sum_{2\leq i<p+\epsilon(p)} c_i t^i +
O(|t|^{p+\epsilon(p)}).
  \end{equation}
\end{thm}
This theorem contains the characteristic expansion of $f\in
L^p$ for $p>1$:
\begin{cor}
Let $f\in L^{1+\eta/2}$, then $f$ admits an accurate
characteristic expansion.
\end{cor}
\begin{proof}
If $f\in L^2$, then \eqref{eqdevelopthighdegree} for $p=2$
becomes $\lambda(t)=1+\ic t E(f)-c t^2/2+o(t^2)$, for some
$c\in \C$. This is the desired characteristic expansion.

Assume now $f\not\in L^2$. Let $\epsilon>0$ be the infimum of
$\epsilon(p)$ for $p\in [1+\eta/2,2]$. Let $p\geq 1+\eta/2$ be
such that $f\in L^p$ and $f\not\in L^{p+\epsilon/2}$. Then
\eqref{eqdevelopthighdegree} gives $\lambda(t)=E(e^{\ic tf})+ c
t^2 + O(|t|^{p+\epsilon})$, which is accurate.
\end{proof}
Together with Lemma \ref{lempasL1}, this concludes the proof of
Proposition \ref{propGM}.

Theorem \ref{thm_highdiff} also contains much more information,
in particular in the $L^2$ case. We will describe in Appendix
\ref{sec_app} another consequence of this very precise
expansion of the eigenvalue $\lambda(t)$, on the speed of
convergence in the central limit theorem. What is remarkable in
that theorem is that the regularity assumption on the function
need not be increased to get finer results, $\sum
m(a)Df(a)^\eta<\infty$ is always sufficient: the only
additional conditions are moment conditions.

\begin{rmk}
For $p<2$, Theorem \ref{thm_highdiff} can be proved using only
the theorem of Keller and Liverani in \cite{keller_liverani},
instead of its extension to several derivatives given in
Paragraph \ref{seckl} (but the resulting bound $\epsilon$ tends
to $0$ when $p$ tends to $2$): in the forthcoming proof, there
is no derivative involved for $p<2$. This gives a more
elementary proof of the accurate characteristic expansion for
functions $f$ not belonging to $L^q$ for some $q<2$, but the
general case (functions in $L^p$ for every $p<2$) requires the
full power of Theorem \ref{thm_1}.
\end{rmk}

We will need the following elementary lemma.
\begin{lem}
\label{lemcontrolFj}
For $j\geq 1$, define a function $F_j:\R \to \C$ by
\begin{equation}
  F_j(x)=e^{\ic x}-\sum_{k=0}^{j-1} (\ic x)^k /k!.
  \end{equation}
Let also $b\in (0,1]$. For $j\geq 1$ and $x\in \R$,
$|F_j(x)|\leq 2 |x|^{j-1+b}$. Moreover, for $j\geq 2$ and
$x,y\in \R$, $|F_j(x)-F_j(y)|\leq 2 |x-y|\cdot
\max(|x|,|y|)^{j-2+b}$.
\end{lem}
\begin{proof}
Let $(A_j)$ denote the property ``for all $x\in \R$,
$|F_j(x)|\leq 2 |x|^{j-1+b}$'' and $(B_j)$ the property ``for
all $x, y$, $|F_j(x)-F_j(y)|\leq 2 |x-y|\cdot
\max(|x|,|y|)^{j-2+b}$''. We claim that $(A_j)$ holds for
$j\geq 1$, and $(B_j)$ holds for $j\geq 2$.

First, $(A_1)$ holds trivially. Moreover, if $(B_j)$ holds,
then $(A_j)$ holds by taking $y=0$. Hence, it is sufficient to
prove that $(A_j)$ implies $(B_{j+1})$ to conclude by
induction. Assume $(A_j)$. Since $F_{j+1}'=\ic F_j$, we have
  \begin{align*}
  |F_{j+1}(x)-F_{j+1}(y)| &\leq |x-y| \sup_{z\in [x,y]} |F_{j+1}'(z)|
  \leq |x-y| \sup_{z\in [x,y]} 2 |z|^{j-1+b}
  \\&
  \leq 2 |x-y| \max(|x|,|y|)^{j-1+b}.
  \end{align*}
This proves $(B_{j+1})$, as desired.
\end{proof}

\begin{proof}[Proof of Theorem~\ref{thm_highdiff}]
Fix once and for all $\eta\in (0,1]$. Let us fix $A>1$, we will
prove the theorem for $p\in [1+1/A,A]$. In this proof, $\delta$
will denote a positive quantity that may only depend on $A$,
and can change from one occurrence to the other.

The quantity $\frac{p}{1+1/5A} -\frac{p}{1+1/2A}$ is bounded
from below, uniformly for $p\in [1+1/A, A]$. Therefore, there
exists $\delta_p \in [1/5A, 1/2A]$ such that the distance from
$p/(1+\delta_p)$ to the integers is $\geq \delta$, for some
$\delta>0$. Let us fix such a $\delta_p$. Let $N\geq 2$ be the
integer such that $N> p/(1+\delta_p)>N-1$, write
$p/(1+\delta_p)=N-1+b_0$ for some $b_0\in [\delta,1-\delta]$.
Let $b_1=\dots=b_{N-1}=1$. Define numbers $p_0,\dots,p_N$ in
$[1+\delta_p,\infty]$ by $p_0=1+\delta_p$ and, for $i\geq 1$,
$p_i=p/(N-i)$. Define also operators $Q_j$ by
$Q_j(v)=\frac{\ic^j}{j!} \Transf(f^j v)$, and let
$\Delta_j(t)=\Transf_t-\Transf-\sum_{k=1}^{j-1} t^k Q_k$. Let
$\tilde \BB_j=L^{p_j}$.

We claim that the assumptions \eqref{eq:Qbound} and
\eqref{eq:Cksmooth} are satisfied for the spaces $\tilde\BB_j$.
Indeed, the choices of $b_0$ and the $p_i$s ensure that, for
$0\leq i<j\leq N$,
  \begin{equation}
  \label{bonspi}
  \frac{1}{p_i}=\frac{1}{p_j}+\frac{b(i,j)}{p}.
  \end{equation}
Therefore, if $u\in L^{p/b(i,j)}$ and $v\in L^{p_j}$, then $uv
\in L^{p_i}$. Since $f^{j-i} \in L^{p/(j-i)}$, this shows that
$Q_{j-i}$ sends $\tilde \BB^j$ to $\tilde \BB^i$ if $i>0$.

By Lemma \ref{lemcontrolFj}, for any $n\geq 1$ and $b>0$,
$|e^{\ic x}-\sum_{k=0}^{n-1} \frac{\ic^k}{k!} x^k|\leq 2
|x|^{n-1+b}$. Therefore, $| \Delta_{j-i}(t) v | \leq 2 \Transf
( |tf|^{j-i-1+b} |v|)$. Taking $b=b_i$, we obtain
  \begin{equation}
  |\Delta_{j-i}(t) v | \leq 2|t|^{b(i,j)}\Transf (|f|^{b(i,j)} |v|).
  \end{equation}
Since $|f|^{b(i,j)}$ belongs to $L^{p/b(i,j)}$, this shows
thanks to \eqref{bonspi} that $\Delta_{j-i}(t)$ sends $\tilde
\BB_j$ to $\tilde \BB_i$ with a norm at most $C |t|^{b(i,j)}$.
This is \eqref{eq:Cksmooth}.

Unfortunately, the spaces $L^{p_j}$ do not satisfy a
Lasota-Yorke type inequality \eqref{eq:lypert2}. Moreover, we
would like to gain a little bit on the integrability exponent.
Therefore, we will rather use spaces $\boL^{q,\reg}$ instead of
spaces $L^q$. To check the assumptions \eqref{eq:Qbound} and
\eqref{eq:Cksmooth}, we will apply Lemma \ref{gain_integrfort}
for some small enough $\beta\in (0, \eta]$ depending only on
$A$.

The assumptions of this lemma are satisfied for the operator
$Q_{j-i}$ ($1\leq i<j\leq N$), with $q=p_j$, $r=p_i$ and
$c=b(i,j)$ (since $f^{j-i}$ is indeed bounded by $|f|^{j-i}$,
and $D(f^{j-i})(a) \leq C Df(a) \norm{1_a
f}_{L^\infty}^{j-i-1}$). We now turn, for $0\leq i< j\leq N$,
to the operators $\Delta_{j-i}(t)$. Once again, we take
$q=p_j$, $r=p_i$ and $c=b(i,j)$. Let us show that the
assumptions of Lemma \ref{gain_integrfort} are satisfied.
First, if $\beta$ is small enough, then $r=p_i$ is larger than
$1+\beta$ (since we have chosen $p_0=1+\delta_p$ with $\delta_p
\geq 1/5A$), and $c=b(i,j)$ is larger than $\beta$ (since
$b_0\geq \delta$ by the good choice of $\delta_p$). Let us
define a function $f_{j-i}(t)=e^{\ic tf} -\sum_{k=0}^{j-i-1}
\frac{(\ic t f)^k}{k!}$, so that $\Delta_{j-i}(t)v=\Transf(
f_{j-i}(t)v)$. The following lemma shows that $f_{j-i}(t)$ is
well behaved, which is the last assumption of Lemma
\ref{gain_integrfort} we have to check.

\begin{lem}
\label{lemujOK} For any $0<b\leq 1$ and $j\geq 1$,
the function $u_j(t)=f_j(t)/(2|t|^{j-1+b})$ satisfies
$|u_j|\leq |f|^{j-1+b}$ and, for all $a\in \alpha$,
  \begin{equation}
  \label{eqborneuj}
  D u_j(t) (a) \leq  \begin{cases}
  Df(a) & \text{ if }j-1+b \leq 1,\\
  Df(a) \norm{1_a f}_{L^\infty}^{j-2+b} & \text{ if }j-1+b>1.
  \end{cases}
  \end{equation}
\end{lem}
\begin{proof}
We have $f_j(t)=F_j(tf)$, where $F_j$ is defined in Lemma
\ref{lemcontrolFj}. Therefore, this lemma  yields $|f_j(t)|\leq
2|tf|^{j-1+b}$ as desired. If $j=1$, $f_j(t)=e^{\ic tf}-1$,
hence \eqref{eqborneuj} follows easily. Assume now $j\geq 2$.
For any points $x,y$ in the same element $a$ of the partition
$\alpha$,
  \begin{align*}
  |f_j(t)(x)-f_j(t)(y)|&=|F_j( tf(x)) -F_j( tf(y))|
  \\&\leq 2 |tf(x)-tf(y)| \max( |tf(x)|, |tf(y)|)^{j-2+b}
  \\&\leq 2|t|^{j-1+b} Df(a) d(x,y) \norm{1_a f}_{L^\infty}^{j-2+b}.
  \end{align*}
This proves \eqref{eqborneuj} in this case.
\end{proof}

Let $\delta>0$ be given by Lemma \ref{gain_integrfort} for the
value of $\beta$ we constructed above. Decreasing $\delta$ if
necessary, we can assume $\delta\leq 1/2A$. Let also
$\reg_N=1$. Lemma \ref{gain_integrfort} (applied to the
operators $Q_1$ and $\Delta_1(t)$, on the space $\boL^{p_N,
\reg_N}$) provides us with $\reg_{N-1}=\reg'$ such that
$\norm{Q_1}_{\boL^{p_N, \reg_N}\to \boL^{p_{N-1}+\delta,
\reg_{N-1}}}$ is finite, and
  \begin{equation}
  \norm{\Delta_1(t)}_{\boL^{p_N, \reg_N}\to \boL^{p_{N-1}+\delta,
  \reg_{N-1}}} = O( |t|^{b(N-1,N)}).
  \end{equation}

Continuing inductively this process, we obtain a sequence
$\reg_N,\reg_{N-1},\dots,\reg_0$ such that, for any $1\leq i <
j\leq N$, the operator $Q_{j-i}$ maps continuously $\boL^{p_j,
\reg_j}$ to $\boL^{p_i+\delta, \reg_i}$, and such that, for any
$0\leq i< j\leq N$, the operator $\Delta_{j-i}(t)$ maps
continuously $\boL^{p_j, \reg_j}$ to $\boL^{p_i+\delta,
\reg_i}$, with a norm at most $C|t|^{b(i,j)}$.

Define a space $\BB_i=\boL^{p_i+\delta, \reg_i}$. Since $\BB_i$
is continuously contained in $\boL^{p_i, \reg_i}$, we have just
proved that the assumptions \eqref{eq:Qbound} and
\eqref{eq:Cksmooth} of Theorem \ref{thm_1} are satisfied.
Moreover, \eqref{eq:lypert1} and \eqref{eq:lypert2} for $M=1$
follow from Lemmas \ref{lemly1} and \ref{lemly2}. Therefore,
Corollary \ref{thm:kl2} applies. Since $\BB_i$ is included in
$L^{p_i+\delta}$, we obtain in particular the following: there
exist $u_1 \in L^{p_{N-1}+\delta},\dots, u_{N-1} \in
L^{p_1+\delta}$ such that the normalized eigenfunction $\xi_t$
of $\Transf_t$ satisfies
  \begin{equation}
  \label{eqexpandxitwell}
  \norm{\xi_t -1 -\sum_{k=1}^{N-1} t^k u_k}_{L^{p_0+\delta}} =O( |t|^{b(0,N) -\epsilon}),
  \end{equation}
for any $\epsilon>0$.

Let us now estimate the eigenvalue $\lambda(t)$ of $\Transf_t$
using this estimate. Let us write $\xi_t-1=\sum_{k=1}^{N-1} t^k
u_k+r_t$, where $r_t$ is an error term controlled by
\eqref{eqexpandxitwell}. By \eqref{eqCentraleLambdat},
  \begin{equation}
  \label{exprimelambdatuk}
  \begin{split}
  \lambda(t)&=E(e^{\ic tf})+\int (e^{\ic tf}-1)(\xi_t-1)
  \\&=E(e^{\ic tf})+\sum_{k=1}^{N-1} t^k \int (e^{\ic tf}-1)u_k
  + \int (e^{\ic tf}-1) r_t.
  \end{split}
  \end{equation}

Let us first estimate $\int (e^{\ic tf}-1) r_t$. We have
$p_0=1+\delta_p$ and $b(0,N)=p/(1+\delta_p)$. Let $q$ be such
that $1/(p_0+\delta)+1/q=1$, i.e., $q=(1+\delta_p+\delta)/
(\delta_p +\delta)$. Since $\delta \leq 1/2A$ and $\delta_p\leq
1/2A$, we obtain $q\geq A$. In particular, $q\geq p$.
Therefore, $|e^{\ic x}-1| \leq 2 |x|^{p/q}$ for any real $x$.
This yields
  \begin{equation}
  \norm{ e^{\ic tf} -1}_ {L^q}\leq \left(\int |e^{\ic tf}-1|^q \right)^{1/q}
  \leq C|t|^{p/q}.
  \end{equation}
Hence,
  \begin{equation}
  \label{lqmskdfjmlqs}
  \left| \int (e^{\ic tf}-1) r_t \right| \leq \norm{ e^{\ic tf} -1}_ {L^q} \norm{r_t}_{L^{p_0+\delta}}
  \leq C|t|^{ p/q + p/(1+\delta_p) -\epsilon}.
  \end{equation}
Moreover,
  \begin{equation}
  \frac{p}{q} + \frac{p}{1+\delta_p} -\epsilon =
  p\left( 1 - \frac{1}{1+\delta_p+\delta} + \frac{1}{1+\delta_p}\right)-\epsilon.
  \end{equation}
Since $\delta$ is positive, this quantity is larger than $p$ if
$\epsilon$ is small enough. Hence, \eqref{lqmskdfjmlqs} is of
the form $O(|t|^{p+\epsilon'})$ for some $\epsilon'>0$. This is
compatible with \eqref{eqdevelopthighdegree}.

We now turn to the terms $t^k \int (e^{\ic tf}-1)u_k$ in
\eqref{exprimelambdatuk}, for $0\leq k\leq N-1$. The function
$u_k$ belongs to $L^{p/k+\delta}$. Let $q$ be such that
$1/q+1/( p/k+\delta)=1$. Let also $c>0$ satisfy $qc=p$. Then
$e^{\ic tf}=\sum_{0\leq j <c} (\ic tf)^j/ j! + r_{c,t}$, where
$|r_{c,t}|\leq 2 |t|^c |f|^c$ by Lemma \ref{lemcontrolFj}. To
conclude the proof, it is sufficient to show that $t^k \int
r_{c,t} u_k=O(|t|^{p+\epsilon'})$ for some $\epsilon'>0$, since
the terms coming from the integrals $t^k \int (\ic tf)^j/ j!
\cdot  u_k$ will contribute to the polynomial in
\eqref{eqdevelopthighdegree}. We have
  \begin{align*}
  \left| t^k \int r_{c,t} u_k \right|&
  \leq |t|^k \norm{r_{c,t}}_{L^q} \norm{u_k}_{L^{p/k+\delta}}
  \leq C |t|^k \left(\int |r_{c,t}|^q \right)^{1/q}
  \\&
  \leq C |t|^{k+c} \left( \int |f|^{p}\right)^{1/q}.
  \end{align*}
Finally, $k+c=k+p-k/(1+k\delta/p)$ is strictly larger than $p$,
since $\delta>0$.
\end{proof}

\begin{rmk}
\label{rmkexpansionmu} When $f\in L^p$, $p>1$, the function
$\mu(t)=\int P_t 1$ appearing in the characteristic expansion
\eqref{eqcarac} of $f$ also satisfies an expansion
  \begin{equation}
  \mu(t)=1+\sum_{1\leq i < p} d_i t^i + O(|t|^{p-\epsilon}),
  \end{equation}
for any $\epsilon>0$. This follows from a similar (but easier)
argument, where one does not need to use the gain in the
exponent from Lemma \ref{gain_integrfort}. This expansion is
not as strong as the expansion of $\lambda(t)$ (it does not
reach the precision $O(|t|^p)$, while Theorem
\ref{thm_highdiff} gets beyond it). The reason for this
difference is that $\mu(t)$ is only expressed in spectral terms
(and Theorem \ref{thm_1} therefore gives a small loss in the
exponent), while for $\lambda(t)$ one can take advantage of the
formula \eqref{eqCentraleLambdat}.
\end{rmk}

\subsection{Last details in the \texorpdfstring{$L^2$}{L\texttwosuperior} case}
\label{par_proof_GM_cobord}
In this paragraph, we conclude the proof of
Theorem~\ref{thm:main}. By Proposition~\ref{propGM} and
Theorem~\ref{thm_main_general}, we only have to identify the
variance $\sigma^2$ when $f\in L^2$, and to strengthen the
conclusion of Theorem~\ref{thm_main_general} in the
$\sigma^2=0$ case.

\begin{lem}
\label{lemdescribevariance} Assume $f\in L^2$, and write $\tilde f=f-\int f$.
Then the asymptotic expansion of $\lambda(t)$ given by Theorem
\ref{thm_highdiff} is
  \begin{equation}
  \lambda(t)=1+\ic t E(f)-(\sigma^2 +E(f)^2) t^2/2+o(t^2),
  \end{equation}
where $\sigma^2=\int \tilde f^2+2\sum_{k=1}^\infty \int \tilde
f\cdot \tilde f\circ T^k$ (the series converges exponentially
fast).
\end{lem}
\begin{proof}
In the expansion \eqref{eqdevelopthighdegree} of $\lambda(t)$,
the term for $i=2$ comes only, in the proof, from the integral
$\int \ic tf u_1$, where
  \begin{equation}
  u_1= \frac{1}{2\ic \pi} \int_{|z-1|=c} (z-\Transf)^{-1} Q_1 (z-\Transf)^{-1} 1 \dd z,
  \end{equation}
where $Q_1(v)=\Transf( \ic f v)$ (and the integral is
converging in a space $\boL^{2+\delta,\reg}$ for some
$\delta>0$ and $\reg>0$). Let us identify $u_1$. We have
$(z-\Transf)^{-1}1= 1/(z-1)$. Moreover, if $E$ is the space of
constant functions and $F$ the space of functions with
vanishing integral, then $(z-\Transf)^{-1}$ is the
multiplication by $1/(z-1)$ on $E$, while
$(z-\Transf)^{-1}v=\sum_{k=0}^\infty z^{-k-1} \Transf^k v$ for
$v\in F$ (the series converging exponentially fast in
$\boL^{2+\delta, \reg}$, and in particular in $L^2$). Writing
$\Transf f$ as $(\int f) +\Transf \tilde f\in E\oplus F$, we
obtain
  \begin{equation}
  u_1= \frac{1}{2\ic \pi} \int_{|z-1|=c} \frac{\ic \int f}{ (z-1)^2}\dd z
  +\sum_{k=0}^\infty \frac{1}{2\ic \pi}\int_{|z-1|=c} \frac{z^{-k-1}}{z-1} \ic \Transf^{k+1} \tilde f \dd z.
  \end{equation}
Since $1/(z-1)^2$ has a vanishing residue at $z=1$, while
$z^{-k-1}/(z-1)$ has a residue equal to $1$, this gives
  \begin{equation}
  u_1= \ic \sum_{k=0}^\infty \Transf^{k+1} \tilde f.
  \end{equation}

We obtain from \eqref{eqdevelopthighdegree}
  \begin{align*}
  \lambda(t)&=E(e^{\ic tf})-t^2 \sum_{k=1}^\infty \int f \Transf^k \tilde f+O(|t|^{2+\epsilon})
  \\&
  =1+\ic t E(f)-t^2 \int f^2 /2 -t^2 \sum_{k=1}^\infty \int \tilde f \cdot \tilde f\circ T^k+o(t^2)
  \\&=1+\ic t E(f) - (\sigma^2+E(f)^2)t^2/2+o(t^2).
  \qedhere
  \end{align*}
\end{proof}

To conclude, it is sufficient to prove that, if $\sigma^2$
vanishes, then $f$ is a bounded coboundary. A similar result is
proved in \cite[Corollary 2.3]{aaronson_denker}, and we will
essentially reproduce the same argument for completeness.

\begin{lem}
Assume $f\in L^2$ is such that $\sigma^2$ (given by
Lemma~\ref{lemdescribevariance}) vanishes. Then there exist a
bounded function $u$ and a real $c$ such that $f=u-u\circ T+c$.
\end{lem}
\begin{proof}
Replacing $f$ with $\tilde f=f-\int f$, we can assume without
loss of generality that $\int f=0$.

The exponential convergence of $\int f\cdot f\circ T^k$ to $0$
ensures that $\int (S_n f)^2 = n\sigma^2+O(1)$. Therefore, if
$\sigma^2=0$, then $S_n f$ is bounded in $L^2$. By Leonov's
Theorem (see e.g.~\cite{aaronson_weiss}), this implies that $f$
is an $L^2$ coboundary: there exists $u\in L^2$ such that
$f=u-u\circ T$ almost everywhere. Then
  \begin{equation}
  \Transf_t(e^{-\ic tu})=\Transf(e^{\ic tf} e^{-\ic tu})=\Transf(e^{-\ic tu\circ
T})=e^{-\ic tu}.
  \end{equation}
By \cite{ionescu_tulcea_marinescu}, this yields $e^{-\ic tu}
\in \boL$. In particular, the function $e^{-\ic tu}$ is
continuous for any small enough $t$.
Lemma~\ref{continuautomatique} shows that $u$ itself is
continuous. In particular, there exists a cylinder
$[b_0,\dots,b_k]$ on which $u$ is bounded. Since $f$ is bounded
on each element of the partition $\alpha$, the equation
$f=u-u\circ T$ implies that $u$ is bounded on $b_k$. Together
with the topological transitivity of $T$, we obtain that $u$ is
bounded on each $a\in \alpha$.

Let $\{a_1,\dots,a_n\}$ be a finite subset of $\alpha$ such
that each element of $\alpha$ contains the image of one of the
$a_i$s (it exists by the big preimage property). Let
$a\in\alpha$, choose $i$ such that $a\subset T(a_i)$, then the
equation $f=u-u\circ T$ gives $\norm{u 1_a}_{L^\infty} \leq
\norm{(|f|+|u|)1_{a_i}}_{L^\infty}$. This shows that $u$ is
uniformly bounded, as desired.
\end{proof}

In fact, a slightly refined version of the same argument also
shows that $u$ is H\"{o}lder continuous.

\begin{lem}
\label{continuautomatique}
Let $u$ be a real function on a metric space $X$, and assume
that $e^{\ic tu}$ is continuous for $t\in [a,b]$ a nontrivial
interval of $\R$. Then $u$ is continuous.
\end{lem}
\begin{proof}
We will show that, if $v_n$ is a real sequence such that
$e^{\ic tv_n}$ converges to $0$ for any $t\in [a,b]$, then
$v_n\to 0$. Applying this result to $v_n=u(x_n)-u(x)$ when
$x_n\to x$, this gives the required continuity of $u$ at $x\in
X$, for any $x$.

Let $A_N=\{ t\in [a,b] \st \forall n\geq N,
\dist(tv_n,2\pi\Z)\leq 1\}$. The set $A_N$ is a closed subset
of $[a,b]$, and $\bigcup A_N=[a,b]$. By Baire's Theorem, there
exists a set $A_N$ containing a nontrivial interval $[c,d]$.
For $n\geq N$ and $t\in [c,d]$, the number $tv_n$ belongs to
$2\pi\Z+[-1,1]$, and depends continuously on $t$. It has to
stay in the same connected component of $2\pi\Z+[-1,1]$,
therefore $| cv_n-dv_n|\leq 2$. This shows that $v_n$ is
bounded.

Any cluster value $v$ of $v_n$ satisfies $e^{\ic tv}=0$ for any
$t\in [a,b]$, hence $v=0$.
\end{proof}

\appendix

\section{The Berry-Esseen theorem for Gibbs-Markov maps}
\label{sec_app}

In this appendix, we obtain necessary and sufficient conditions
for the Berry-Esseen theorem, for Gibbs-Markov maps.

\begin{thm}
\label{thm_be} Let $T:X \to X$ be a probability preserving
mixing Gibbs-Markov map, and let $f:X\to \R$ satisfy
$\sum_{a\in\alpha} m(a)Df(a)^\eta<\infty$ for some $\eta\in
(0,1]$. Assume $f\in L^2$ and $E(f)=0$, and $S_n f/\sqrt{n}\to
\boN(0,\sigma^2)$ with $\sigma^2>0$. Let
  \begin{equation}
  \Delta_n\coloneqq \sup_{x\in \R} \left|m\{ S_n f/\sqrt{n} < x\} - P ( \boN(0,\sigma^2)< x)\right|.
  \end{equation}
Let $\delta\in (0,1)$. Then $\Delta_n= O( n^{-\delta/2})$ if
and only if $E( f^2 1_{|f|> x})=O(x^{-\delta})$ when
$x\to\infty$.

Moreover, $\Delta_n= O( n^{-1/2})$ if and only if $E( f^2
1_{|f|> x})=O(x^{-1})$ when $x\to\infty$, and $E(f^3
1_{|f|<x})$ is uniformly bounded.
\end{thm}
When one considers i.i.d.~random variables instead of Birkhoff
sums, this theorem for $\delta<1$ is proved in \cite[Theorem
3.4.1]{ibragimov_linnik}, and the proof for $\delta=1$ is given
in \cite{ibragimov}. For the proof in the dynamical setting, we
will essentially follow the same strategy as in the
i.i.d.~case, the additional crucial ingredient being the
estimate on $\lambda(t)$ provided by Theorem
\ref{thm_highdiff}. We will only give the proof for $\delta<1$,
since the proof for $\delta=1$ is very similar following the
arguments of \cite{ibragimov}.

\begin{proof}[Proof of the necessity in Theorem \ref{thm_be}]
Assuming $\Delta_n=O(n^{-\delta/2})$, we will prove $E( f^2
1_{|f|> x})=O(x^{-\delta})$. This is trivial if $f\in L^3$, so
we can assume this is not the case. In this proof, $\epsilon$
will denote the minimum of $\epsilon(p)$ given by Theorem
\ref{thm_highdiff} for $p\in [2,3]$. Consider $p\in [2,3]$ such
that $f\in L^p$ and $f\not\in L^{p+\epsilon/2}$, and let
$q=\min(p+\epsilon,3)$. Hence, $\lambda(t)=E(e^{\ic tf})+
ct^2+O(|t|^q)$ for some $c\in \R$. It will be more convenient
to write this estimate as follows:
  \begin{equation}
  \label{eqdecritlambd}
  \lambda(t)=E(e^{\ic tf}) e^{ct^2 + t^2 \phi(t)}\text{ with } \phi(t)=O(|t|^{q-2}).
  \end{equation}

Let $W$ be the symmetrization of $f$, i.e., the difference of
two independent copies of $f$. Its characteristic function is
$E(e^{\ic tW})=|E(e^{\ic tf})|^2$. Let us write $E(e^{\ic
tW})=e^{-\sigma_0^2 t^2 + t^2 \gamma_0(t)}$ where
$\sigma_0^2=E(f^2)$ and $\gamma_0$ is a real function defined
on a neighborhood of $0$. \cite[Paragraph
III.4]{ibragimov_linnik} proves the following fact:
  \begin{multline}
  \label{eq_IL}
  \text{If }\int_0^x t^2 |\gamma_0(t)|=O(x^{3+\tilde\delta}),\
  0<\tilde\delta<1,
  \text{ when }x\to 0,\\
  \text{ then }E( f^2
  1_{|f|> x})=O(x^{-\tilde\delta}) \text{ when }x\to+\infty.
  \end{multline}
To conclude, it is therefore sufficient to estimate $\int t^2
|\gamma_0(t)|$.

Let $H$ denote the distribution function of
$\boN(0,2\sigma^2)$, and $F_n$ the distribution function of the
difference of two independent copies of $S_n f/\sqrt{n}$. From
the assumption $\Delta_n=O(n^{-\delta/2})$, it follows that
$\sup_{x\in\R} |H(x)-F_n(x)|\leq C n^{-\delta/2}$. Let $h(t)$
and $f_n(t)$ be the characteristic functions of $H$ and $F_n$,
i.e., $h(t)=e^{-\sigma^2 t^2}$ and $f_n(t)=|E(e^{\ic tS_n
f/\sqrt{n}})|^2$. Integrating by parts the equality
$f_n(t)-h(t)=\int e^{\ic tx} \dd( F_n(x)-H(x))$, we obtain
  \begin{equation}
  \frac{f_n(t)-h(t)}{\ic t}=\int e^{\ic tx} (F_n(x)-H(x)) \dd x.
  \end{equation}
This shows that the $L^2$ functions $(f_n(t)-h(t))/\ic t$ and
$F_n-H$ are Fourier transforms of one another. The functions $t
e^{-t^2/2}$ and $-\ic x e^{-x^2/2}/ \sqrt{2\pi}$ are also
Fourier transforms of one another. Hence, Parseval's theorem
gives
  \begin{equation}
  \int \frac{f_n(t)-h(t)}{t} \cdot t e^{-t^2/2}\dd t= C \int (F_n(x)-H(x)) x e^{-x^2/2}
  = O (n^{-\delta/2}).
  \end{equation}
Since $\int_{|t|\geq \log n} e^{-t^2/2}\dd t=O(n^{-\delta/2})$,
this yields
  \begin{equation}
  \label{qksjlfmkqsjdf}
  \int_{|t|\leq \log n} (f_n(t)-h(t)) e^{-t^2/2} \dd t=O (n^{-\delta/2}).
  \end{equation}

The characteristic expansion of $f$ gives
  \begin{equation*}
  f_n(t)=|E(e^{\ic tS_n f/\sqrt{n}})|^2= \left| \lambda\left(\frac{t}{\sqrt{n}}\right)\right|^{2n}
  \left|\mu\left(\frac{t}{\sqrt{n}}\right)\right|^2+\epsilon_n(t),
  \end{equation*}
where $\epsilon_n(t)$ tends exponentially fast to $0$ (by
Theorem \ref{thm:kato}), and the function $\mu$ satisfies
$\mu(t)=1+O(t)$ (by Remark \ref{rmkexpansionmu}). Let
$g_n(t)=\left|
\lambda\left(\frac{t}{\sqrt{n}}\right)\right|^{2n}$, then
$\int_{|t|\leq \log n} (f_n(t)-g_n(t)) e^{-t^2/2} \dd
t=O(n^{-1/2})$. Therefore, \eqref{qksjlfmkqsjdf} gives
  \begin{equation}
  \int_{|t|\leq \log n} (g_n(t)-h(t)) e^{-t^2/2} \dd t=O (n^{-\delta/2}).
  \end{equation}
Moreover, by \eqref{eqdecritlambd}
  \begin{align*}
  g_n(t)&=|E(e^{\ic t f/\sqrt{n}})|^{2n} e^{2ct^2 +2t^2 \Real \phi(t/\sqrt{n})}
  \\&=e^{-\sigma_0^2 t^2 +t^2 \gamma_0(t/\sqrt{n}) + 2ct^2 + 2t^2 \Real \phi(t/\sqrt{n})}
  \\&=e^{-\sigma^2 t^2 + t^2 \gamma_0(t/\sqrt{n}) + 2t^2 \Real \phi(t/\sqrt{n})}.
  \end{align*}
Let $h_n(t)=e^{-\sigma^2 t^2 + t^2 \gamma_0(t/\sqrt{n})}$.
Since $\phi(t)=O(|t|^{q-2})$ by \eqref{eqdecritlambd}, we have
$\int_{|t|\leq \log n} (g_n(t)-h_n(t))  e^{-t^2/2} \dd t =
O(n^{- (q-2)/2})$. Hence,
  \begin{equation}
  \int_{|t|\leq \log n} (h_n(t)-h(t)) e^{-t^2/2} \dd t=O(n^{-\delta/2})+O(n^{-(q-2)/2}).
  \end{equation}
Since $h_n(t)-h(t)=e^{-\sigma^2 t^2} (e^{t^2
\gamma_0(t/\sqrt{n})}-1)$, we can now conclude as in \cite[Page
106]{ibragimov_linnik} to get
  \begin{equation}
  \int_0^x t^2 |\gamma_0(t)|=O(x^{3+\delta})+O(x^{q+1}).
  \end{equation}
By \eqref{eq_IL}, this proves that $E(f^2
1_{|f|>x})=O(x^{-\tilde \delta})$ for $\tilde
\delta=\min(\delta, q-2)$. If $q-2<\delta$ (in particular,
$q\not=3$, so $q=p+\epsilon$), we have $\tilde\delta=q-2$,
hence $f$ belongs to $L^{q'}$ for any $q'<q$. In particular,
$f\in L^{q-\epsilon/2}=L^{p+\epsilon/2}$. This is not
compatible with the choice of $p$. Hence, $q-2\geq \delta$,
whence $\tilde\delta=\delta$, and $E(f^2
1_{|f|>x})=O(x^{-\delta})$.
\end{proof}

\begin{proof}[Proof of the sufficiency in Theorem \ref{thm_be}]
Assuming $E( f^2 1_{|f|> x})=O(x^{-\delta})$, we will prove
$\Delta_n=O(n^{-\delta/2})$. We essentially follow the
arguments of the proof of the necessity, in the reverse
direction, the main difference being that we do not need any
more to work with the symmetrization of the random variables.

Let us write $E(e^{\ic tf})=e^{-\sigma_0^2 t^2/2 + t^2
\gamma(t)}$. \cite[Page 111]{ibragimov_linnik} proves that,
under the assumption $E( f^2 1_{|f|> x})=O(x^{-\delta})$, the
function $\gamma$ satisfies $\int_0^x t^2 |\gamma(t)|\dd
t=O(x^{3+\delta})$. Moreover, $f$ belongs to $L^p$ for any
$p<2+\delta$. Let $q=\min(2+\delta+\epsilon/2, 3)>2+\delta$.
Taking $p=2+\delta-\epsilon/2$, Theorem \ref{thm_highdiff}
shows that $\lambda(t)=E(e^{\ic tf})+ct^2+O(|t|^q)$, which we
may rewrite as $\lambda(t)=E(e^{\ic tf}) e^{ct^2+ t^2 \phi(t)}$
where $\phi(t)=O(|t|^{q-2})$. Together with the expansion of
$E(e^{\ic tf})$, we obtain
  \begin{equation}
  \label{A10}
  \lambda(t)=e^{-\sigma^2 t^2/2 + t^2 \psi(t)} \text{ with }\int_0^x t^2 |\psi(t)|\dd t=O(x^{3+\delta}).
  \end{equation}

Let $f_n$ denote the characteristic function of $S_n
f/\sqrt{n}$. The classical Berry-Esseen estimate \cite[Theorem
1.5.2]{ibragimov_linnik} shows that, for any $T>0$,
  \begin{equation}
  \Delta_n\leq C \int_{-T}^T \frac{1}{|t|} |f_n(t)-e^{-\sigma^2 t^2/2}| \dd t+ C/T.
  \end{equation}
Let us choose $T=\rho \sqrt{n}$ with $\rho$ small enough. The
second term in this estimate is then
$O(n^{-1/2})=O(n^{-\delta})$. For the first term, we split the
integral in two parts, corresponding to $|t|\leq 1/n$ and
$|t|>1/n$. In the first part, we have
  \begin{equation}
  |f_n(t)-1|=\left|E(e^{\ic t S_n f/\sqrt{n}} -1)\right|\leq |t| E|S_n f|/\sqrt{n} \leq \sqrt{n} |t|.
  \end{equation}
The resulting integral is bounded by
  \begin{equation}
  \int_{|t|\leq 1/n} \sqrt{n} + |t|^{-1} |1-e^{-\sigma^2 t^2/2}| \dd t=O(n^{-1/2}).
  \end{equation}
Hence,
  \begin{equation}
  \Delta_n\leq C \int_{1/n\leq |t|\leq \rho\sqrt{n}} \frac{1}{|t|} |f_n(t)-e^{-\sigma^2 t^2/2}| \dd t+ O(n^{-1/2}).
  \end{equation}
We have
$f_n(t)=\lambda(t/\sqrt{n})^n\mu(t/\sqrt{n})+\epsilon_n(t)$,
where $\epsilon_n(t)$ tends exponentially fast to $0$, while
$\mu(t)=1+O(t)$. Let $g_n(t)=\lambda(t/\sqrt{n})^n$. By
\eqref{A10}, if $\rho$ is small enough, we have
$|\lambda(t)|\leq e^{-\sigma^2 t^2/4}$ for $|t|\leq \rho$. This
yields $|\lambda(t/\sqrt{n})|^n \leq e^{-\sigma^2 t^2/4}$ for
$|t|\leq \rho\sqrt{n}$. Hence,
  \begin{equation}
  \int_{1/n\leq |t|\leq \rho\sqrt{n}} \frac{1}{|t|} |f_n(t)-g_n(t)|\leq C/\sqrt{n}.
  \end{equation}
With \eqref{A10}, we obtain
  \begin{align*}
  \Delta_n&\leq C \int_{1/n\leq |t|\leq \rho\sqrt{n}} \frac{1}{|t|} |g_n(t)-e^{-\sigma^2 t^2/2}| \dd t+ O(n^{-1/2})
  \\ &= C \int_{1/n\leq |t|\leq \rho\sqrt{n}} \frac{1}{|t|} e^{-\sigma^2 t^2/2} | e^{t^2 \psi(t/\sqrt{n})} -1| \dd t
  + O(n^{-1/2}).
  \end{align*}
Since $\int_0^x t^2 |\psi(t)|\dd t=O(x^{3+\delta})$, this last
integral is bounded by $O(n^{-\delta/2})$ (see
e.g.~\cite[bottom of Page 107]{ibragimov_linnik}). This
concludes the proof.
\end{proof}

\bibliography{biblio}
\bibliographystyle{amsalpha}

\end{document}